\documentclass[12pt]{amsart}
\usepackage{latexsym}
\usepackage{color}
\usepackage{amsmath}
\usepackage{amsfonts}
\newtheorem{theorem}{Theorem}[section]
\newtheorem{lemma}[theorem]{Lemma}
\newtheorem{proposition}[theorem]{Proposition}

\theoremstyle{definition}

\theoremstyle{remark}

\numberwithin{equation}{section}

\thanks{$^*$ Corresponding author: Kunio Hidano}
\begin{document}
\allowdisplaybreaks
\title[Global solutions to quasi-linear wave equations]
{Global existence for null-form wave equations 
with data in a Sobolev space of lower regularity and weight}

\author[K.\,Hidano]{Kunio Hidano$^*$}
\author[K.\,Yokoyama]{Kazuyoshi Yokoyama}
\address{Department of Mathematics, Faculty of Education, 
Mie University, 1577 Kurima-machiya-cho Tsu, Mie Prefecture 514-8507, Japan}
\email{hidano@edu.mie-u.ac.jp}
\address{Hokkaido University of Science, 
7-Jo 15-4-1 Maeda, Teine, Sapporo, 
Hokkaido 006-8585, Japan}
\email{yokoyama@hus.ac.jp}


\renewcommand{\thefootnote}{\fnsymbol{footnote}}
\footnote[0]{2020\textit{ Mathematics Subject Classification}.
 Primary 35L52, 35L15; Secondary 35L72}




\begin{abstract}
Assuming initial data have small weighted 
$H^4\times H^3$ norm, 
we prove global existence of solutions 
to the Cauchy problem for systems of 
quasi-linear wave equations in three space dimensions 
satisfying the null condition of Klainerman. 
Compared with the work of Christodoulou, 
our result assumes smallness of data 
with respect to $H^4\times H^3$ norm 
having a lower weight.
Our proof uses 
the space-time $L^2$ estimate 
due to Alinhac for some special derivatives of solutions 
to variable-coefficient wave equations. 
It also uses the conformal energy estimate 
for inhomogeneous wave equation $\Box u=F$. 
A new observation made in this paper is that, 
in comparison with the proofs of Klainerman and H\"ormander, 
we can limit the number of occurrences 
of the generators of hyperbolic rotations or dilations 
in the bootstrap argument. 
This limitation allows us to obtain global solutions 
for radially symmetric data, 
when a certain norm with considerably 
lower weight is small enough. 
\end{abstract}

\maketitle

\section{Introduction}
This paper is concerned with the Cauchy problem for 
systems of quasi-linear wave equations 
\begin{equation}\label{eq1}
\Box u_i
+
F_i(\partial u,\partial^2 u)
+
C_i(u,\partial u,\partial^2 u)=0,
\quad
t>0,\,\,x\in{\mathbb R}^3
\end{equation}
$(i=1,\dots,N$ for some $N\in{\mathbb N})$ with 
initial data 
\begin{equation}\label{data1}
u_i(0)=f_i,\quad
\partial_t u_i(0)=g_i.
\end{equation}
Here, $\Box=\partial_t^2-\Delta$, 
$u=(u_1,\dots,u_N)$, 
$\partial u=(\partial u_1,\dots,\partial u_N)$, 
$\partial u_i=(\partial_0 u_i,\dots,\partial_3 u_i)$, 
$\partial^2 u=(\partial^2 u_1,\dots,\partial^2 u_N)$, 
$\partial^2 u_i=(\partial_0^2 u_i,\partial_0\partial_1 u_i,\dots,
\partial_3^2 u_i)$, 
$\partial_0=\partial/\partial t$, 
$\partial_a=\partial/\partial x_a$ $(a=1,2,3)$. 
As in the seminal papers \cite{Christodoulou} and \cite{KlainermanNull86}, 
we will discuss the diagonal system, 
and we suppose that the quadratic nonlinear term 
$F_i(\partial u,\partial^2 u)$ has the form 
\begin{equation}\label{nlt1}
F_i(\partial u,\partial^2 u)
=
F_i^{j,\alpha\beta\gamma}(\partial_\gamma u_j)(\partial_{\alpha\beta}^2 u_i)
+
F_i^{jk,\alpha\beta}(\partial_\alpha u_j)(\partial_\beta u_k), 
\quad
i=1,\dots,N
\end{equation}
for real constants $F_i^{j,\alpha\beta\gamma}$ and $F_i^{jk,\alpha\beta}$. 
Here, and in the following, 
we use the summation convention, 
that is, 
if lowered and raised, 
repeated indices of 
Greek letters and Roman letters are summed 
from $0$ to $3$ and $1$ to $N$, respectively. 
As for the higher-order term 
$C_i(u,\partial u,\partial^2 u)$, 
we may suppose without loss of generality 
that 
it is cubic because this paper is concerned only with small solutions. 
We thus suppose it has the form 
\begin{equation}\label{nlt2}
C_i(u,\partial u,\partial^2 u)
=
G_i^{\alpha\beta}(u,\partial u)\partial_{\alpha\beta}^2 u_i
+
H_i(u,\partial u),\quad
i=1,\dots,N,
\end{equation}
where 
$G_i^{\alpha\beta}(u,v)$ is a homogeneous polynomial 
of degree 2, and $H_i(u,v)$ is a homogeneous polynomial 
of degree $3$ in $u$ and $v$. 
Since we consider $C^2$-solutions, 
we may suppose without loss of generality
\begin{equation}\label{sym1}
F_i^{j,\alpha\beta\gamma}
=
F_i^{j,\beta\alpha\gamma},
\quad
G_i^{\alpha\beta}(u,\partial u)
=
G_i^{\beta\alpha}(u,\partial u).
\end{equation}
For given $i,j$, 
we say that the set of the coefficients 
$\{F_i^{j,\alpha\beta\gamma}:\alpha,\beta,\gamma=0,\dots, 3\}$ 
satisfies the null condition if we have 
\begin{equation}\label{null1}
F_i^{j,\alpha\beta\gamma}
X_\alpha X_\beta X_\gamma=0
\end{equation}
for any $X=(X_0,X_1,X_2,X_3)\in{\mathbb R}^4$ 
satisfying $X_0^2=X_1^2+X_2^2+X_3^2$. 
Also, for given $i,j,k$, we say that 
the set of the coefficients 
$\{F_i^{jk,\alpha\beta}:\alpha,\beta=0,\dots,3\}$ 
satisfies the null condition if we have 
\begin{equation}\label{null2}
F_i^{jk,\alpha\beta}
X_\alpha X_\beta=0
\end{equation}
for any $X=(X_0,X_1,X_2,X_3)\in{\mathbb R}^4$ 
satisfying $X_0^2=X_1^2+X_2^2+X_3^2$. 
We say that 
the system (\ref{eq1}) satisfies the null condition 
if the sets of the coefficients $\{F_i^{j,\alpha\beta\gamma}\}$ 
and $\{F_i^{jk,\alpha\beta}\}$ 
satisfy the null condition for all given $i,j$ and $i,j,k$, 
respectively.

For the scalar wave equations, thus $N=1$ in (\ref{eq1}), 
with a quadratic nonlinear term 
$\Box u=(\partial_t u)^2$ and 
$\Box u=|\nabla u|^2$, 
nonexistence of global smooth solutions 
was shown even for small data by 
John \cite{John1981} and Sideris \cite{Sideris1983}, 
respectively. 
Actually, in these two papers, nonexistence of 
global solutions was shown also 
for some types of quasi-linear wave equations. 
On the other hand, if the system (\ref{eq1}) satisfies the null condition 
and the initial data is sufficiently small, smooth, and 
compactly supported, 
then the Cauchy problem (\ref{eq1})--(\ref{data1}) admits 
a unique global smooth solution. 
This was shown by Klainerman \cite{KlainermanNull86} with use of 
the generators of the Lorentz transformations and the dilations, 
in addition to the standard partial differential operators 
$\partial_\alpha$ (the generators of translations). 
The conformal energy, the Klainerman inequality (see (\ref{kl1}) below), 
and his $L^1$--$L^\infty$ 
weighted estimate for inhomogeneous wave equations \cite{Kl1984}, 
which are all written in terms of these generators, 
played an important role in his proof. 
(To be precise, an earlier version of (\ref{kl1}) was employed in 
\cite{KlainermanNull86}.) 
Later, H\"ormander \cite{Hor} refined the $L^1$--$L^\infty$ weighted estimate 
of Klainerman and showed 
\begin{equation}\label{hor1}
(1+t+|x|)|u(t,x)|
\leq
C
\sum_{j+|a|+|b|+|c|+d\leq 2}
\int_0^t\!\!\int_{{\mathbb R}^3}
\frac{|\partial_s^j\partial_y^a\Omega^b L^c S^d F(s,y)|}{1+s+|y|}\,dyds
\end{equation}
for the equation $\Box u=F$ with zero data. 
Making an effective use of (\ref{hor1}), 
he gave a more precise assumption on 
smallness of data. 
Namely, H\"ormander gave an alternative proof 
of global existence under the weaker assumption that 
the quantities related with the given initial data 
\begin{align}
&
\sum_{i=1}^N\sum_{j+|a|+|b|+|c|+d\leq 5}
(1+|x|)
|\partial_t^j\partial_x^a\Omega^b L^cS^d u_i(0,x)|,
\label{small1}\\
&\sum_{i=1}^N\sum_{j+|a|+|b|+|c|+d\leq 9}
\|\partial_t^j\partial_x^a\Omega^b L^cS^d u_i(0)\|_{L^2({\mathbb R}^3)}
\label{small2}
\end{align}
are small enough. 
(The definition of $\Omega^b L^cS^d$ is given below. 
We remark that for $j=2,\dots, 9$, we can calculate 
$\partial_t^j u_i(0,x)$ with the help of the equation (\ref{eq1}), 
thus these two quantities are determined by the given small data. 
We also remark that, by virtue of the Sobolev type 
inequality (\ref{sob25}), the smallness of (\ref{small2}) actually 
ensures that of (\ref{small1}).) 
On the other hand, in \cite{Christodoulou}, 
Christodoulou assumed smallness of data 
with respect to the weighted $H^4\times H^3$ norm 
\begin{equation}\label{small3}
\sum_{i=1}^N
\biggl(
\sum_{|a|\leq 4}
\|\langle x\rangle^{3+|a|}\partial_x^a f_i\|_{L^2({\mathbb R}^3)}
+
\sum_{|a|\leq 3}
\|\langle x\rangle^{4+|a|}\partial_x^a g_i\|_{L^2({\mathbb R}^3)}
\biggr)
\end{equation}
and proved global existence result 
under the null condition by the conformal mapping method. 
(Here, and in the following as well, we employ the standard notation 
$\langle x\rangle:=\sqrt{1+|x|^2}$.) 
In comparison with this Christodoulou's size condition, 
a question naturally arises\,: does the method of 
using the generators yield the proof of 
global existence of solutions to (\ref{eq1})--(\ref{data1}) 
under the null condition when some weighted 
$H^4\times H^3$ norm of data is small enough? 
Exploiting a new way of handling the null-form quadratic nonlinear terms 
with use of the weighted $L^2$ estimate 
for some special derivatives, 
Alinhac proved his truly remarkable energy inequality and 
gave an affirmative answer to this long-standing problem, 
in the special case where 
all the cubic terms $C_i(u,\partial u,\partial^2 u)$ are absent. 
See pages 92--94 in \cite{Al2010}. 
In this connection, 
we cite Theorem 1.4 of \cite{HZ2019} here. 
(We remark that 
the theorem of Alinhac on page 94 in \cite{Al2010} 
was slightly improved in \cite{HY2017}. 
See Theorem 1.5 there. This result in \cite{HY2017} 
was then slightly improved in \cite{HZ2019}.) 
\begin{theorem}\label{theoremhz}
Suppose that $(\ref{eq1})$ satisfies the null condition 
and that in $(\ref{nlt2})$, 
\begin{equation}\label{nlt3}
G_i^{\alpha\beta}(u,v)\equiv 0,\quad
H_i(u,v)\equiv 0
\end{equation}
for every $i$, $\alpha$, and $\beta$. 
Then there exists $\varepsilon>0$ such that 
if $f_i\in L^6({\mathbb R}^3)$ $(i=1,\dots,N)$ and 
\begin{equation}\label{small4}
\sum_{i=1}^N
\sum_{{|a|+|b|+d\leq 3}\atop{d\leq 1}}
\biggl(
\|
\partial_x\partial_x^a\Omega^b\Lambda^d f_i
\|_{L^2({\mathbb R}^3)}
+
\|
\partial_x^a\Omega^b\Lambda^d g_i
\|_{L^2({\mathbb R}^3)}
\biggr)
<\varepsilon
\quad
(\Lambda:=x\cdot\nabla),
\end{equation}
then the unique local $($in time$)$ solution 
to $(\ref{eq1})$ can be continued globally in time. 
\end{theorem}
We remark that Theorem 1.4 in \cite{HZ2019} is concerned with 
the scalar equation, i.e., (\ref{eq1}) with $N=1$, 
but obviously the method there is general enough to 
prove Theorem \ref{theoremhz} above. 
The proof of this theorem is carried out by the combination of 
the ghost weight energy method of Alinhac \cite{Al2010} with 
the Klainerman-Sideris method \cite{KS}. 
We must enhance the discussions in \cite{Al2010}, 
\cite{HY2017}, and \cite{HZ2019} 
concerning the special case (\ref{nlt3}), 
because the system (\ref{eq1}) contains $u$ itself 
in the nonlinear terms. 
It is a natural attempt to inject into the argument in \cite{Al2010}, 
\cite{HY2017}, and \cite{HZ2019} 
such key elements of the proof due to 
Klainerman \cite{KlainermanNull86} 
and H\"ormander \cite{Hor} 
as the conformal energy and the $L^1$--$L^\infty$ 
weighted estimate for inhomogeneous wave equations (\ref{hor1}). 
Combining these elements with the ghost weight energy method in 
\cite{Al2010} and \cite{HY2017}, 
we can indeed obtain the following\,:
\begin{proposition}\label{alinhacpro}
Suppose that the system $(\ref{eq1})$ satisfies the null condition.
There exist positive constants $C$ and $\varepsilon$ such that 
if 
\begin{equation}\label{small5}
\sum_{i=1}^N
\biggl(
\sum_{|a|\leq 4}
\|
\langle x\rangle^{|a|}\partial_x^a f_i
\|_{L^2({\mathbb R}^3)}
+
\sum_{|a|\leq 3}
\|
\langle x\rangle^{|a|+1}\partial_x^a g_i
\|_{L^2({\mathbb R}^3)}
\biggr)
<\varepsilon,
\end{equation}
then the Cauchy problem {\rm $(\ref{eq1})$--$(\ref{data1})$} admits 
a unique global solution satisfying 
\begin{align}\label{sol1}
\sum_{i=1}^N
&
\biggl\{
\|(1+t+|x|)u_i\|_{L^\infty({\mathbb R}^+\times{\mathbb R}^3)}
+
\sum_{|a|\leq 3}
\|\partial\Gamma^a u_i\|_{L^\infty({\mathbb R}^+;L^2({\mathbb R}^3))}\\
&
+
\sum_{j=1}^3
\sum_{|a|\leq 3}
\biggl(
\int_0^\infty
\biggl\|
\frac{1}{(1+|s-|\cdot||)^{(1/2)+\eta}}
T_j\Gamma^a u_i(s,\cdot)
\biggr\|_{L^2({\mathbb R}^3)}^2ds
\biggr)^{1/2}\nonumber\\
&
+
\sum_{|a|\leq 2}
\|(1+t)^{-\delta}
\Gamma^a u_i\|_{L^\infty({\mathbb R}^+;L^2({\mathbb R}^3))}\nonumber\\
&
+
\sum_{|a|=3}
\|(1+t)^{-2\delta}
\Gamma^a u_i(t)\|_{L^\infty({\mathbb R}^+;L^2({\mathbb R}^3))}
\biggr\}
\leq
C\varepsilon.\nonumber
\end{align}
Here ${\mathbb R}^+:=(0,\infty)$, 
$\delta$ and $\eta$ are sufficiently small positive constants 
and $T_j=\partial_j+(x_j/|x|)\partial_t$, $j=1,2,3$. 
\end{proposition}
Concerning the definition of the commonly used operators $\Gamma^a$, 
see, e.g., \cite[p.\,46]{John1990}, 
\cite[p.\,301]{KlainermanNull86}. 
Namely, 
$$
\sum_{|a|\leq 3}
\|\partial\Gamma^a u_i\|_{L^\infty({\mathbb R}^+;L^2({\mathbb R}^3))}
=
\sum_{j+|a|+|b|+|c|+d\leq 3}
\|\partial\partial_t^j\partial_x^a\Omega^b L^cS^d 
u_i\|_{L^\infty({\mathbb R}^+;L^2({\mathbb R}^3))}
$$
and so on. 
In comparison with the Christodoulou's size condition (\ref{small3}), 
the above one (\ref{small5}) has an advantage; 
it obviously assumes less decay on the data. 
Compared with (\ref{small5}), however, 
the size condition (\ref{small4}), 
which though applies to the special case (\ref{nlt3}), 
has an advantage that if 
$f_i$ and $g_i$ are radially symmetric 
(hence $\Omega f_i=\Omega g_i=0$) 
and the norm with the low weight
\begin{align}\label{small6}
\sum_{i=1}^N&
\biggl(
\|\partial_x f_i\|_{L^2({\mathbb R}^3)}
+
\|g_i\|_{L^2({\mathbb R}^3)}
\biggr)\\
&
+
\sum_{i=1}^N\sum_{1\leq |a|\leq 3}
\biggl(
\|
\langle x\rangle\partial_x\partial_x^a f_i
\|_{L^2({\mathbb R}^3)}
+
\|
\langle x\rangle\partial_x^a g_i
\|_{L^2({\mathbb R}^3)}
\biggr)\nonumber
\end{align}
is small enough, then 
(\ref{eq1})--(\ref{data1}) admits global solutions. 
In view of the current state of the art, 
the purpose of this paper is to 
show global existence of small solutions to (\ref{eq1})--(\ref{data1}) 
under the null condition 
when 
initial data have lower regularity than was assumed in 
\cite{KlainermanNull86} and \cite{Hor}, 
and have weaker decay than was assumed in \cite{Christodoulou} 
and Proposition \ref{alinhacpro}. 
In particular, taking into account 
Theorem \ref{theoremhz} which holds for the special case (\ref{nlt3}), 
we would naturally like to obtain global solutions for radially symmetric data 
when a low weight norm of data is small enough. 
Recall $\Lambda:=x\cdot\nabla$. 
We define
\begin{align}\label{small7}
D(f,g)
:&=
\sum_{i=1}^N
\biggl\{
\sum_{|a|+|b|\leq 3}
\bigl(
\|\partial_x\partial_x^a\Omega^b f_i\|_{L^2({\mathbb R}^3)}
+
\|\partial_x^a\Omega^b g_i\|_{L^2({\mathbb R}^3)}
\bigr)\\
&
\hspace{1.4cm}
+
\sum_{|a|+|b|\leq 2}
\bigl(
\|\partial_x\partial_x^a\Omega^b\Lambda f_i\|_{L^2({\mathbb R}^3)}
+
\|\partial_x^a\Omega^b\Lambda g_i\|_{L^2({\mathbb R}^3)}
\bigr)\nonumber\\
&
\hspace{1.4cm}
+
\|f_i\|_{L^2({\mathbb R}^3)}\nonumber\\
&
\hspace{1.4cm}
+
\sum_{|b|\leq 2}
\bigl(
\|\Omega^b\Omega f_i\|_{L^2({\mathbb R}^3)}
+
\|\Omega^b\Lambda f_i\|_{L^2({\mathbb R}^3)}
+
\||x|\Omega^b g_i\|_{L^2({\mathbb R}^3)}
\bigr)
\biggr\}.\nonumber
\end{align}
Now we are in a position to state our main theorem. 
\begin{theorem}\label{maintheorem} 
Suppose that the system $(\ref{eq1})$ satisfies the null condition. 
There exist positive constants $C$, $\varepsilon$ such that if 
$D(f,g)<\varepsilon$, 
then the Cauchy problem {\rm $(\ref{eq1})$--$(\ref{data1})$} admits 
a unique global solution satisfying 
\begin{align}\label{sol2}
\sum_{i=1}^N
\biggl\{&
\sum_{{|a|+|b|+|c|+d\leq 3}
      \atop
      {|c|+d\leq 1}
     }
\|\partial\partial_x^a\Omega^b L^cS^d
u_i\|_{L^\infty({\mathbb R}^+;L^2({\mathbb R}^3))}
+
\sum_{|b|\leq 2}
\|\Omega^b u_i\|_{L^\infty({\mathbb R}^+;L^2({\mathbb R}^3))}\\
&
+
\sum_{|b|\leq 2}
\biggl(
\|\langle t\rangle^{-\delta}
(S+2)\Omega^b u_i\|_{L^\infty({\mathbb R}^+;L^2({\mathbb R}^3))}\nonumber\\
&
\hspace{1.4cm}
+
\|\langle t\rangle^{-\delta}\Omega\Omega^b
u_i\|_{L^\infty({\mathbb R}^+;L^2({\mathbb R}^3))}
+
\|\langle t\rangle^{-\delta}
L\Omega^b u_i\|_{L^\infty({\mathbb R}^+;L^2({\mathbb R}^3))}
\biggr)\nonumber\\
&
+
\sum_{k=1}^3
\biggl(
\sum_{
       {|a|+|b|+|c|+d\leq 3}
       \atop
       {|c|+d\leq 1}
      }
\|
\langle t-r\rangle^{-(1/2)-\eta}
T_k\partial_x^a\Omega^b L^c S^d u_i
\|_{L^2((0,\infty)\times{\mathbb R}^3)}\nonumber\\
&
\hspace{1.5cm}
+\sum_{
       {|a|+|b|+|c|+d\leq 2}
       \atop
       {|c|+d\leq 1}
      }
\|
\langle t-r\rangle^{-(1/2)-\eta}
T_k\partial_t\partial_x^a\Omega^b L^c S^d u_i
\|_{L^2((0,\infty)\times{\mathbb R}^3)}
\biggr)\nonumber\\
&
+\sum_{{|a|+|b|+|c|+d\leq 3}
      \atop
      {|c|+d\leq 1}
     }
\biggl(
\sup_{t>0}
\langle t\rangle^{-(1/4)-\delta/2}
\|
|x|^{-5/4}
\partial_x^a \Omega^b L^c S^d u_i
\|_{L^2((0,t)\times{\mathbb R}^3)}\nonumber\\
&
\hspace{2.9cm}
+
\sup_{t>0}
\langle t\rangle^{-(1/4)-\delta/2}
\|
|x|^{-1/4}
\partial\partial_x^a \Omega^b L^c S^d u_i
\|_{L^2((0,t)\times{\mathbb R}^3)}
\biggr)
\biggr\}\nonumber\\
&
\leq C\varepsilon.\nonumber
\end{align}
Here, $\delta$, and $\eta$ are positive constants satisfying 
$\delta<1/6$, $\eta<1/3$.
\end{theorem}
Let $\varepsilon$ be sufficiently small. 
We easily see that such an oscillating and decaying data 
as $u(0,x)=\varepsilon (\sin x_1)\langle x\rangle^{-d}$ 
can be allowed in the theorem of 
Christodoulou if $d>17/2$, while 
Theorem \ref{maintheorem} above allows the smaller values of $d$, 
that is, $d>9/2$. 
Also, we benefit from the size condition (\ref{small7}) and obtain 
global solutions when $f_i$ and $g_i$ are radially symmetric and 
the low weight norm 
\begin{equation}\label{small8}
\sum_{i=1}^N
\biggl(
\|f_i\|_{L^2({\mathbb R}^3)}
+
\sum_{1\leq |a|\leq 4}
\|\langle x\rangle \partial_x^a f_i\|_{L^2({\mathbb R}^3)}
+
\sum_{|a|\leq 3}
\|\langle x\rangle \partial_x^a g_i\|_{L^2({\mathbb R}^3)}
\biggr)
\end{equation}
is small enough. 
It means that such an oscillating and more slowly decaying 
radially symmetric data 
as $u(0,x)=\varepsilon (\sin\langle x\rangle)\langle x\rangle^{-d}$ 
with $d>5/2$ is allowed in Theorem \ref{maintheorem}.

The new size condition (\ref{small7}), where 
the number of occurrences of $\Lambda$ is limited 
at most to $1$ in the norms there, 
is a direct consequence of 
the limitation of 
that of occurrences of $S$ in the norms (\ref{sol2}). 
Also, in (\ref{small7}) we are allowed 
to employ the low weight norms to measure the size of data, 
which results from the limitation of the number of 
occurrences of $L_j$ in the norms (\ref{sol2}). 
While the $L^1$--$L^\infty$ estimate (\ref{hor1}) and the Klainerman inequality 
\cite{Kl87}
\begin{equation}\label{kl1}
(1+t+|x|)(1+|t-|x||)^{1/2}|u(t,x)|
\leq
C\sum_{{j+|a|+|b|}\atop{+|c|+d\leq 2}}
\|\partial_t^j\partial_x^a\Omega^b L^c S^d u(t)\|_{L^2({\mathbb R}^3)}
\end{equation}
play an important role in the proof of Proposition \ref{alinhacpro}, 
we encounter $L^cS^d$ with $|c|+d=2$ in (\ref{hor1}) and (\ref{kl1}), 
and therefore must refrain from using these two well-known inequalities 
in the proof of Theorem \ref{maintheorem}. 
To get over this difficulty, 
we will exploit the effective idea of 
estimating nonlinear terms over the set 
$\{x\in{\mathbb R}^3:|x|<(t+1)/2\}$ and its complement 
(for any fixed $t>0$) separately. 
(See, e.g., \cite{LiYu}, \cite{SiderisTu}, and \cite{H2016} 
for earlier papers using this simple but important idea.) 
As a consequence, 
some simple Sobolev-type or trace-type inequalities (\ref{sob21})--(\ref{sob25}) and 
(\ref{hy}), 
combined with the weighted space-time $L^2$ estimate (\ref{rod}) 
and the Li-Yu estimate (\ref{ly}), 
play a role as the good substitute for (\ref{hor1}) and (\ref{kl1}). 
Using the ghost weight energy inequality for 
variable-coefficient wave equations and the conformal energy estimate 
for the standard wave equation $\Box u=F$ 
together with these good substitutes, 
we will prove Theorem \ref{maintheorem}. 

We end this section with setting the notation in this paper 
and giving attention to how to 
handle the cubic terms $C_i(u,\partial u,\partial^2 u)$ 
at the stage of carrying out 
the conformal-energy type estimate for local solutions. \\
\noindent{\bf Notation.} 
We use the operators 
$\partial_\alpha$ $(\alpha=0,\dots,3)$, 
$\Omega_{ij}:=x_i\partial_j-x_j\partial_i$ 
$(1\leq i<j\leq 3)$, 
$L_k:=x_k\partial_t+t\partial_k$ 
$(k=1,2,3)$, 
and $S:=t\partial_t+x\cdot\nabla$. 
We set 
$\partial_x^a:=\partial_1^{a_1}\partial_2^{a_2}\partial_3^{a_3}$, 
$\Omega^b:=\Omega_{12}^{b_1}\Omega_{13}^{b_2}\Omega_{23}^{b_3}$, 
and 
$L^c:=L_1^{c_1}L_2^{c_2}L_3^{c_3}$ 
for $a=(a_1,a_2,a_3)$, 
$b=(b_1,b_2,b_3)$, and 
$c=(c_1,c_2,c_3)$, 
respectively. 
In this paper, we denote 
$\partial_1$, 
$\partial_2$, 
$\partial_3$, 
$\Omega_{12}$, 
$\Omega_{23}$, 
$\Omega_{13}$, 
$L_1$, $L_2$, $L_3$, and $S$ 
by 
$Z_1, Z_2,\dots, Z_{10}$. 
The notation 
\begin{equation}\label{zbar}
{\bar Z}^a:=\partial_1^{a_1}\partial_2^{a_2}\partial_3^{a_3}
\Omega_{12}^{a_4}\Omega_{13}^{a_5}\Omega_{23}^{a_6},
\quad a=(a_1,\dots,a_6)
\end{equation}  
is used repeatedly. 
We set $\Lambda:=x\cdot\nabla=r\partial_r$. 

We define the energy and its associated quantity
\begin{align}
&
E_1(v(t))
:=
\frac12
\int_{{\mathbb R}^3}
\bigl(
(\partial_t v(t,x))^2
+
|\nabla v(t,x)|^2
\bigr)dx,
\quad
N_1(v(t)):=\sqrt{E_1(v(t))},\label{e1}\\
&
N_{j+1}(v(t))
:=
\biggl(
\sum_{{|a|+|b|+|c|+d\leq j}\atop{|c|+d\leq 1}}
E_1(\partial_x^a \Omega^b L^c S^d v(t))
\biggr)^{1/2},\quad j=1,2,3.\label{n4}
\end{align}
Moreover, we use the conformal energy 
\begin{align}
&
Q(v(t))\label{conf1}\\
&
:=
\frac12
\int_{{\mathbb R}^3}
\biggl(
((S+2)v(t,x))^2
+
\sum_{|b|=1}
(\Omega^b v(t,x))^2
+
\sum_{|c|=1}
(L^c v(t,x))^2
\biggr)dx.\nonumber
\end{align}
We mention the important fact that the inequality 
\begin{equation}\label{conf2}
\|v(t)\|_{L^2({\mathbb R}^3)}^2
\leq
CQ(v(t))
\end{equation}
holds for a positive constant $C$. 
For the proof, 
see \cite[pp.\,311--322]{KlainermanNull86}. 
See also \cite[pp.\,101--102]{Hor} or \cite[pp.\,98--101]{AlinhacText}, 
where a different proof can be found. 
By (\ref{conf2}), it is easy to verify 
the equivalence of $Q(v(t))$ and 
$$
{\tilde Q}(v(t))
:=
\|v(t)\|_{L^2({\mathbb R}^3)}^2
+
\sum_{|b|=1}\|\Omega^b v(t)\|_{L^2({\mathbb R}^3)}^2
+
\sum_{|c|=1}\|L^c v(t)\|_{L^2({\mathbb R}^3)}^2
+
\|Sv(t)\|_{L^2({\mathbb R}^3)}^2.
$$
We set 
\begin{equation}\label{mj+1}
M_{j+1}(v(t))
:=
\biggl(
\sum_{|a|+|b|\leq j}
{\tilde Q}(\partial_x^a\Omega^b v(t))
\biggr)^{1/2},\quad j=0,1,2.
\end{equation}
We also need
\begin{equation}\label{xj}
X_j(v(t))
:=
\biggl(
\sum_{
      |a|+|b|\leq j
      }
\|
\partial_x^a
\Omega^b v(t)
\|_{L^2({\mathbb R}^3)}^2
\biggr)^{1/2},\quad j=0,1,2. 
\end{equation}
For ${\mathbb R}^N$-valued functions 
$w(t,x)=(w_1(t,x),\dots,w_N(t,x))$, 
we set
\begin{align}
&
E_1(w(t))
=
\sum_{i=1}^N E_1(w_i(t)),
\quad
{\tilde Q}(w(t))
=
\sum_{i=1}^N {\tilde Q}(w_i(t)),\\
&
N_{j+1}(w(t))
=
\biggl(
\sum_{i=1}^N
N_{j+1}(w_i(t))^2
\biggr)^{1/2},\quad j=1,2,3.
\end{align}
$M_{j+1}(w(t))$ and $X_j(w(t))$ $(j=0,1,2)$ 
are defined similarly. 
We obviously have 
$X_0(w(t))\leq M_1(w(t))$, 
$X_j(w(t))
\leq
N_j(w(t))+M_j(w(t))$, $j=1,2$. 
Let us recall that 
careful attention should be paid 
not only on the quadratic null-form terms 
but also on the cubic terms 
$C_i(u,\partial u,\partial^2 u)$, 
especially at the stage where 
the estimate of the conformal energy is carried out. 
Since the weighted norm of the forcing term $F$ 
appears on the right-hand side of (\ref{confest}) below, 
the cubic terms as well as the quadratic null-form terms 
are regarded as ``critical'' ones, 
and therefore a mildly growing (in time) bound 
for the conformal energy is the most that one can obtain 
in general. 
The bootstrap argument of 
\cite{KlainermanNull86}, \cite{Hor} successfully 
employed such a {\it weak} bound for the conformal energy 
in combination with 
a {\it sharp} point-wise decay estimate 
obtained with the use of linear estimates such as (\ref{hor1}), 
when handling the cubic terms and closing the estimates. 
See, e.g., \cite[p.\,141]{Hor}, 
especially the sentences:
``In each term we can estimate all factors 
except one using (6.6.30). 
For the third order terms $f_I^2$ and $f_I^4$ 
this gives a factor 
$\leq C\varepsilon^2(1+x_0+|{\overrightarrow x}|)^{-2}$ 
in addition to a factor with norm square 
$O(E_k(u;x_0))$''. 
In place of such a sharp point-wise decay estimate, 
our bootstrap argument employs 
a sharp estimate for local solutions in the $X_2$ norm 
obtained with the use of the Li-Yu estimate 
(see (\ref{ly}) below). 
Actually, the use of the $X_2$ norm 
is one of the crucial ingredients in order to limit 
the number of occurrences 
of the generators of hyperbolic rotations or dilations 
in the bootstrap argument. 
(Another key ingredient is to use the weighted space-time 
$L^2$ norm for the estimate of the energy-type norm 
$N_4(u(t))$. See (\ref{genergyest}).) 
The $X_2$ norm is effectively used, 
especially at the stage where 
the cubic terms are handled 
in the course of carrying out the conformal-energy type estimate 
for local solutions (\ref{m3conf313}). 
See, in particular, 
the terms 
$\int_0^t\langle\tau\rangle^{-1}X_2(u(\tau))^2
N_4(u(\tau))d\tau$
and  $\int_0^t\langle\tau\rangle^{-1}X_2(u(\tau))^3 d\tau$ 
on the right-hand side of (\ref{m3conf313}) below, 
where the $X_2$ norm plays a crucial role. 
Indeed, if we employed 
$N_2(u(\tau))+M_2(u(\tau))$ there in place of $X_2(u(\tau))$ 
(recall that one always has 
$X_2(u(\tau))\leq N_2(u(\tau))+M_2(u(\tau)$)), 
we could not close the estimates. 

This paper is organized as follows. 
In the next section, we first recall some special properties that 
the null-form nonlinear terms enjoy, 
and then we recall several key inequalities 
that play an important role in our arguments. 
In Section \ref{sect3}, 
we carry out the energy estimate, 
following the ghost weight energy method of Alinhac.  
Sections \ref{sectionconformal} and \ref{sectionx2norm} 
are devoted to obtaining bounds for 
$M_3(u(t))$ and $X_2(u(t))$, respectively. 
In Section \ref{sectionspacetimel2norm}, 
we carry out the $L^2$ weighted space-time estimate, 
using the Keel-Smith-Sogge type estimate. 
In the final section, we complete the proof of Theorem \ref{maintheorem}. 
\section{Preliminaries}
The proof of our theorems builds on several lemmas. 
Let $[\cdot,\cdot]$ stand for the commutator\,: 
$[A,B]:=AB-BA$. 
\begin{lemma}\label{lemmacomm}
The following commutation relations hold 
for $1\leq j<k\leq 3$, $l=1,2,3$, and $\alpha=0,\dots,3:$ 
\begin{align}
&
[\Omega_{jk},\Box]=0, \quad
[L_l,\Box]=0, \quad
[S,\Box]=-2\Box,\label{comm1}\\
&
[S,\Omega_{jk}]=0,\quad
[S,L_l]=0,\quad
[S,\partial_\alpha]=-\partial_\alpha,\label{comm2}\\
&
[L_l,\Omega_{jk}]
=
\delta_{lj}L_k-\delta_{lk}L_j.\label{comm3}
\end{align}
We also have for $l,j=1,2,3$
\begin{equation}\label{comm4}
[L_l,\partial_t]
=
-\partial_l,
\quad
[L_l,\partial_j]
=
-\delta_{lj}\partial_t.
\end{equation}
Furthermore, we have for 
$1\leq j<k\leq 3,\,1\leq l<m\leq 3$
\begin{equation}\label{comm5}
[\Omega_{jk},\Omega_{lm}]
=
\delta_{kl}\Omega_{jm}
+
\delta_{km}\Omega_{lj}
+
\delta_{jl}\Omega_{mk}
+
\delta_{jm}\Omega_{kl} 
\end{equation} 
and for $1\leq j<k\leq 3,\,l=1,2,3$
\begin{equation}\label{comm6}
[\Omega_{jk},\partial_l]
=
-\delta_{lj}\partial_k
+
\delta_{lk}\partial_j.
\end{equation}
\end{lemma}

Recall that in this paper, we denote 
$\partial_1$, 
$\partial_2$, 
$\partial_3$, 
$\Omega_{12}$, 
$\Omega_{23}$, 
$\Omega_{13}$, 
$L_1$, $L_2$, $L_3$, and $S$ 
by 
$Z_1, Z_2,\dots, Z_{10}$. 
The next lemma states that 
the null condition is preserved 
under the differentiation. 
\begin{lemma}\label{nullpreserved}
Suppose that for given $j$, 
the coefficients $F^{j,\alpha\beta\gamma}$ satisfy the null condition. 
Also, suppose that for given $j,k$ the coefficients 
$F^{jk,\alpha\beta}$ satisfy the null condition. 
Then, for any $Z_l$ $(l=1,\dots,10)$ we have
\begin{align}
Z_l&
F^{j,\alpha\beta\gamma}
(\partial_\gamma v)
(\partial_{\alpha\beta}^2 w)\label{null21}\\
&=
F^{j,\alpha\beta\gamma}
(\partial_\gamma Z_l v)
(\partial_{\alpha\beta}^2 w)
+
F^{j,\alpha\beta\gamma}
(\partial_\gamma v)
(\partial_{\alpha\beta}^2 Z_l w)
+
{\tilde F}_l^{j,\alpha\beta\gamma}
(\partial_\gamma v)
(\partial_{\alpha\beta}^2 w)\nonumber
\end{align}
holds 
with the new coefficients 
${\tilde F}_l^{j,\alpha\beta\gamma}$ 
also satisfying the null condition. 
Also, the equality 
\begin{align}
Z_l&
F^{jk,\alpha\beta}
(\partial_\alpha v)
(\partial_\beta w)\label{null22}\\
&
=
F^{jk,\alpha\beta}
(\partial_\alpha Z_l v)
(\partial_\beta w)
+
F^{jk,\alpha\beta}
(\partial_\alpha v)
(\partial_\beta Z_l w)
+
{\tilde F}_l^{jk,\alpha\beta}
(\partial_\alpha v)
(\partial_\beta w)\nonumber
\end{align}
holds 
with the new coefficients 
${\tilde F}_l^{jk,\alpha\beta}$ 
also satisfying the null condition. 
\end{lemma} 
For the proof, see, e.g., \cite[p.\,91]{Al2010}. 

The next lemma can be shown 
essentially in the same way as in \cite[pp.\,90--91]{Al2010}. 
Together with it, we will later exploit the fact that 
for local solutions $u$, 
the special derivatives $T_i u$ have 
better space-time $L^2$ integrability 
and improved time decay property of their $L^\infty({\mathbb R}^3)$ norms. 
\begin{lemma}\label{lemmanullt}
Suppose that for every $i,\,j$, and $k$, 
the coefficients $F_i^{j,\alpha\beta\gamma}$ 
and $F_i^{jk,\alpha\beta}$ satisfy the null condition. 
Then, we have 
for smooth functions $w_i(t,x)$ $(i=1,2,3)$
\begin{align}
&
|
F_i^{j,\alpha\beta\gamma}
(\partial_\gamma w_1)
(\partial_{\alpha\beta}^2 w_2)
|
\leq
C
\bigl(
|T w_1|
|\partial^2 w_2|
+
|\partial w_1|
|T\partial w_2|
\bigr),\label{null23}\\
&
|
F_i^{j,\alpha\beta\gamma}
(\partial_{\alpha\gamma}^2 w_1)
(\partial_\beta w_2)
|
\leq
C
\bigl(
|T\partial w_1|
|\partial w_2|
+
|\partial^2 w_1|
|T w_2|
\bigr),\label{null24}\\
&
|
F_i^{j,\alpha\beta\gamma}
(\partial_\gamma w_1)
(\partial_\beta w_2)
(\partial_\alpha w_3)
|,\,
|
F_i^{j,\alpha\beta\gamma}
(\partial_\gamma w_1)
(\partial_\beta w_2)
(-\omega_\alpha)
(\partial_t w_3)
|\label{null25}\\
&
\hspace{0.2cm}
\leq
C
\bigl(
|T w_1|
|\partial w_2|
|\partial w_3|
+
|\partial w_1|
|T w_2|
|\partial w_3|
+
|\partial w_1|
|\partial w_2|
|T w_3|
\bigr),\nonumber\\
&
|
F_i^{jk,\alpha\beta}
(\partial_\alpha v)
(\partial_\beta w)
|
\leq
C
\bigl(
|T v|
|\partial w|
+
|\partial v|
|T w|
\bigr).\label{null26}
\end{align}
\end{lemma}
Here, and in the following, we use the notation 
$\omega_0=-1$, $\omega_k=x_k/|x|$, $k=1,2,3$. 
Also, for $v$ and $\partial v=(\partial_0v,\dots,\partial_3v)$, 
we use 
\begin{equation}\label{tv}
|Tv|
:=
\biggl(
\sum_{k=1}^3
|T_k v|^2
\biggr)^{1/2},\quad
|T\partial v|
:=
\biggl(
\sum_{k=1}^3
\sum_{\gamma=0}^3
|T_k\partial_\gamma v|^2
\biggr)^{1/2},
\end{equation}
where $T_k=\partial_k+\omega_k\partial_t$, as before. 

Inspired by \cite{Al2010}, 
we also use the remarkable improvement 
of point-wise decay of 
the special derivatives $T_kv(t,x)$. 
\begin{lemma}[\cite{Al2010}, pp.\,90--91]
The inequalities 
\begin{align}
&|Tv(t,x)|
\leq
\frac{C}{|x|}
\biggl(
\sum_{|b|=1}
|\Omega^b v(t,x)|
+
\sum_{|c|=1}
|L^c v(t,x)|
+
|Sv(t,x)|
\biggr),\label{tv1}\\
&
|Tv(t,x)|
\leq
\frac{C}{t}
\biggl(
\sum_{|b|=1}
|\Omega^b v(t,x)|
+
\sum_{|c|=1}
|L^c v(t,x)|
+
|Sv(t,x)|
\biggr)\label{tv2}
\end{align}
hold for smooth functions $v(t,x)$. 
\end{lemma}
This is a direct consequence of 
\begin{align}
&T_k
=
\frac{1}{r}
\bigl(
L_k+(r-t)\partial_k
\bigr)
=
\frac{1}{t}
\bigl(
L_k+(t-r)\omega_k\partial_t
\bigr),\label{tk}\\
&
\partial_k
=
\frac{tL_k+\displaystyle{\sum_{j=1}^3}x_j\Omega_{kj}-x_kS}{t^2-r^2},
\quad
\partial_t
=
\frac{tS-\displaystyle{\sum_{j=1}^3}x_jL_j}{t^2-r^2}.\label{ddd}
\end{align}
We thus omit the proof of (\ref{tv1}). 

Next, let us show some Sobole-type or trace-type inequalities. 
In the following, we use the notation $\partial_r:=(x/|x|)\cdot\nabla$, 
and for $p\in [1,\infty]$ and $q\in [1,\infty)$
\begin{align}
&
\|v\|_{L_r^\infty L_\omega^p({\mathbb R}^3)}
:=
\sup_{r>0}
\|v(r\cdot)\|_{L^p(S^2)},\label{norm21}\\
&
\|v\|_{L_r^q L_\omega^p({\mathbb R}^3)}
:=
\biggl(
\int_0^\infty \|v(r\cdot)\|_{L^p(S^2)}^q r^2dr
\biggr)^{1/q}.\label{norm22}
\end{align}
\begin{lemma}\label{manysob}
Suppose that $v$ decays sufficiently fast as $|x|\to\infty$. 
Then, we have
\begin{align}
&
\|\langle t-r \rangle v(t)\|_{L^6({\mathbb R}^3)}
\leq
C\sum_{|a|+|b|+|c|+d\leq 1}
\|\partial_x^a \Omega^b L^c S^d v(t)\|_{L^2({\mathbb R}^3)}
\label{sob21}\\
&
\langle t-r \rangle |v(t,x)|
\leq
C\|v(t)\|_{H^2({\mathbb R}^3)}
+
C\sum_{|a|\leq 1}\sum_{|b|+|c|+d=1}
\|\partial_x^a \Omega^b L^c S^d v(t)\|_{L^2({\mathbb R}^3)}.\label{sob22}
\end{align}
Moreover, we have
\begin{align}
&
\|r^{1/2}v(t)\|_{L_r^\infty L_\omega^4}
\leq
C
\|\nabla v(t)\|_{L^2({\mathbb R}^3)},\label{sob23}\\
&
\|rv(t)\|_{L_r^\infty L_\omega^4}
\leq
C
\|\partial_r v\|_{L^2({\mathbb R}^3)}^{1/2}
\biggl(
\sum_{|b|\leq 1}
\|\Omega^b v(t)\|_{L^2({\mathbb R}^3)}
\biggr)^{1/2},\label{sob24}\\
&
r|v(t,x)|
\leq
C
\biggl(
\sum_{|b|\leq 1}
\|\partial_r\Omega^b v(t)\|_{L^2({\mathbb R}^3)}
\biggr)^{1/2}
\biggl(
\sum_{|b|\leq 2}
\|\Omega^b v(t)\|_{L^2({\mathbb R}^3)}
\biggr)^{1/2}.\label{sob25}
\end{align}
\end{lemma}
\begin{proof}
For the proof of (\ref{sob21}), we first employ 
the well-known inequality 
$\|w\|_{L^6({\mathbb R}^3)}\leq C\|\nabla w\|_{L^2({\mathbb R}^3)}$ 
for $w=\langle t-r \rangle v(t,x)$ and 
then use the first equality in (\ref{ddd}) to obtain 
\begin{equation}\label{trd}
|t-r||\partial_k v(t,x)|
\leq
|L_k v(t,x)|
+
\sum_{j=1}^3|\Omega_{kj}v(t,x)|
+
|Sv(t,x)|.
\end{equation}
For the proof of (\ref{sob22}), 
we apply the Sobolev embedding 
$W^{1,6}({\mathbb R}^3)\hookrightarrow L^\infty({\mathbb R}^3)$ 
to the function $\langle t-r \rangle v(t,x)$ 
and then use (\ref{sob21}). 
For the proof of the first trace-type inequality (\ref{sob23}), 
see, e.g., \cite[(3.16)]{Sideris2000}.  
For the proof of the second trace-type inequality (\ref{sob24}), 
see, e.g., \cite[(3.19)]{Sideris2000}. 
The Sobolev embedding  
$W^{1,4}(S^2)\hookrightarrow L^\infty(S^2)$ 
together with (\ref{sob21}) immediately yields (\ref{sob25}). 
\end{proof}
\begin{lemma}\label{hyinequality}
Suppose that 
$v$ decays sufficiently fast $|x|\to\infty$. 
For any $\theta\in [0,1/2]$, 
there exists a constant $C>0$ such that 
the inequality 
\begin{equation}\label{hy}
r^{(1/2)+\theta}
\langle t-r \rangle^{1-\theta}
\|v(t,r\cdot)\|_{L^4(S^2)}
\leq
C\sum_{{|a|+|b|}\atop{+|c|+d\leq 1}}
\|
\partial_x^a \Omega^b L^c S^d v(t)
\|_{L^2({\mathbb R}^3)}
\end{equation} 
holds. 
\end{lemma}
\begin{proof}
For $\theta=1/2$, we first follow the proof of \cite[(3.19)]{Sideris2000} 
with $\beta=0$ and then use (\ref{trd}) above. 
For $\theta=0$, we first apply (\ref{sob23}) above to the function 
$w(t,x)=\langle t-r \rangle v(t,x)$ and 
then use (\ref{trd}). 
We follow the idea in Section 2 of \cite{MNS-JJM2005} 
and obtain (\ref{hy}) for $\theta\in (0,1/2)$ by interpolation. 
\end{proof}
To bound local solutions in the $M_3$ norm, 
we employ the conformal energy estimate 
(see, e.g., \cite[Theorem 6.11]{AlinhacText}) 
together with the equivalence of $Q(v(t))$ and 
${\tilde Q}(v(t))$. 
\begin{lemma} The solution $u$ to the inhomogeneous wave equation 
$\partial_t^2 u-\Delta u=F$ in 
${\mathbb R}^3\times (0,\infty)$ 
with data $(f,g)$ at $t=0$ satisfies 
the conformal energy estimate\,$:$
\begin{align}\label{confest}
\sum_{|b|=1}&
\|
\Omega^b u(t)
\|_{L^2({\mathbb R}^3)}
+
\sum_{|c|=1}
\|
L^c u(t)
\|_{L^2({\mathbb R}^3)}
+
\|(S+2)u(t)\|_{L^2({\mathbb R}^3)}\\
&
\leq
C(
\|f\|_{L^2({\mathbb R}^3)}
+
\||x|\nabla f\|_{L^2({\mathbb R}^3)}
+
\|
|x|g
\|_{L^2({\mathbb R}^3)}
)\nonumber\\
&
\hspace{0.2cm}
+
C\int_0^t
\|
(\tau+|x|)F(\tau)
\|_{L^2({\mathbb R}^3)}
d\tau.\nonumber
\end{align}
\end{lemma} 
The following estimate is essentially due to Li and Yu \cite{LiYu}, 
and we employ it to bound local solutions in the $X_2$ norm. 
\begin{lemma}
The solution $u$ to the inhomogeneous wave equation 
$\partial_t^2 u-\Delta u=F$ in 
${\mathbb R}^3\times (0,\infty)$ 
with data $(f,g)$ at $t=0$ satisfies
\begin{align}\label{ly}
\|&u(t)\|_{L^2({\mathbb R}^3)}
\leq
\|f\|_{L^2({\mathbb R}^3)}
+
\||D|^{-1}g\|_{L^2({\mathbb R}^3)}
+
C\int_0^t\|\chi_1 F(\tau)\|_{L^{6/5}({\mathbb R}^3)}d\tau\\
&
\hspace{2.3cm}
+
C\int_0^t
\langle\tau\rangle^{-1/2}
\|\chi_2 F(\tau)\|_{L_r^1 L_\omega^{4/3}({\mathbb R}^3)}d\tau.\nonumber
\end{align}
The functions $\chi_1$ and $\chi_2$ denote 
the characteristic functions of 
$\{x\in{\mathbb R}^3:|x|<(1+\tau)/2\}$ and 
$\{x\in{\mathbb R}^3:|x|>(1+\tau)/2\}$, respectively. 
\end{lemma}
\begin{proof}
This is a consequence of the Sobolev inequality 
$\|v\|_{L^6({\mathbb R}^3)}\leq C\|\nabla v\|_{L^2({\mathbb R}^3)}$, 
the trace inequality (\ref{sob23}) above, and the duality argument. 
\end{proof}
We also need the space-time $L^2$ estimates for 
the variable-coefficient operator $P$ defined as 
\begin{equation}\label{pdef}
P
:=
\partial_t^2-\Delta
+h^{\alpha\beta}(t,x)\partial_{\alpha\beta}^2
=
\bigl(
-m^{\alpha\beta}+h^{\alpha\beta}(t,x)
\bigr)\partial_{\alpha\beta}^2. 
\end{equation}
Here, $(m^{\alpha\beta})=\mbox{{\rm diag}}\,(-1,1,1,1)$, 
and the  variable coefficients 
$h^{\alpha\beta}\in C^\infty((0,T)\times{\mathbb R}^3)$  
$(\alpha, \beta = 0,1,2,3)$ satisfy the symmetry condition 
$h^{\alpha\beta}=h^{\beta\alpha}$. 
Define the (modified) energy-momentum tensor as 
\begin{equation}
T^{\alpha\beta}
:=
m^{\alpha\mu}
(m^{\beta\nu}-h^{\beta\nu})
(\partial_\mu v)
(\partial_\nu v)
-\frac12
m^{\alpha\beta}
(m^{\mu\nu}-h^{\mu\nu})
(\partial_\mu v)
(\partial_\nu v).
\end{equation}
A straightforward computation yields\,:
\begin{lemma}\label{energymomentum}
Let $g=g(\rho)\in C^1({\mathbb R})$. 
The equality 
\begin{align}
\partial&_\beta
(e^{g(t-r)}T^{0\beta})\\
&
-
e^{g(t-r)}
\biggl(
(\partial_t v)
(Pv)
+
(\partial_\beta h^{\beta\nu})
(\partial_t v)
(\partial_\nu v)
-
\frac12
(\partial_t h^{\mu\nu})
(\partial_\mu v)
(\partial_\nu v)
\biggr)\nonumber\\
&
\hspace{0.4cm}
-
e^{g(t-r)}
g'(t-r)(-\omega_\beta)T^{0\beta}=0\nonumber
\end{align}
holds. 
Here, as in Lemma $\ref{lemmanullt}$, 
$\omega_0=-1$, 
$\omega_k=x_k/|x|$, $k=1,2,3.$
\end{lemma}
In the next section, we employ Lemma \ref{energymomentum} 
together with 
\begin{equation}\label{2019june26t00}
T^{00}
=
\frac12
|\partial_t v|^2
+
\frac12
|\nabla v|^2
+
h^{0\nu}(\partial_t v)(\partial_\nu v)
-\frac12
h^{\mu\nu}
(\partial_\mu v)(\partial_\nu v)
\end{equation}
and 
\begin{equation}\label{2019june26t0beta}
(-\omega_\beta)T^{0\beta}
=
\frac12
\sum_{j=1}^3
(T_j v)^2
-
\omega_\beta
h^{\beta\nu}
(\partial_t v)
(\partial_\nu v)
-
\frac12
h^{\mu\nu}
(\partial_\mu v)
(\partial_\nu v)
\end{equation}
to obtain the Alinhac type $L^2$ weighted space-time estimate for 
the special derivatives 
$T_k u_i$ of local solutions $u=(u_1,\dots,u_N)$. 
Though employing the ghost weight energy method of Alinhac, 
we mention that it is possible to get 
essentially the same $L^2$ weighted space-time estimate 
by following the idea of Lindblad and Rodnianski \cite{LR2005}, 
Lindblad, Nakamura, and Sogge \cite[Lemma A.1]{LNS2013}. 

We also need the Keel-Smith-Sogge type $L^2$ weighted space-time estimate for 
the standard derivatives. 
In addition to the symmetry condition, 
we further suppose the size condition $\sum |h^{\alpha\beta}(t,x)|\leq 1/2$. 
Owing to the method in \cite[Appendix]{Ster}, 
we have the following\,:
\begin{lemma}[Theorem 2.1 of \cite{HWY2012Adv}]\label{spacetimeL2}
For $0<\mu<1/2$, there exists a positive constant $C$ such that 
the inequality 
\begin{align}
(1&+T)^{-2\mu}
\left(
\|
r^{-(3/2)+\mu}u
\|^2_{L^2((0,T)\times{\mathbb R}^3)}
+
\|
r^{-(1/2)+\mu}
\partial u
\|^2_{L^2((0,T)\times{\mathbb R}^3)}
\right)\label{rod}\\
&
\leq
C\|\partial u(0,\cdot)\|^2_{L^2({\mathbb R}^3)}\nonumber\\
&
\hspace{0.4cm}
+C
\int_0^T\!\!\!\int_{{\mathbb R}^3}
\left(
|\partial u||Pu|
+
\frac{|u||Pu|}{r^{1-2\mu}\langle r\rangle^{2\mu}}
+
|\partial h||\partial u|^2 \right. \nonumber \\
&
\hspace{2.9cm}
\left.
+ 
\frac{|\partial h||u\partial u|}{r^{1-2\mu}\langle r\rangle^{2\mu}}
+ 
\frac{|h||\partial u|^2}{r^{1-2\mu}\langle r\rangle^{2\mu}}
+ 
\frac{|h||u \partial u|}{r^{2 - 2\mu}\langle r\rangle^{2\mu}} 
\right)dxdt\nonumber
\end{align}
holds for smooth and compactly supported 
$($for any fixed time$)$ functions $u(t,x)$.
\end{lemma}
\section{Bound for $N_4(u(t))$}\label{sect3}
The proof of local (in time) existence due to H\"ormander 
for scalar wave equations (see \cite[Theorem 6.4.11]{Hor}) is obviously valid for 
the systems (\ref{eq1}) under consideration. 
In what follows, 
we always assume that for a small constant ${\hat \varepsilon}$, 
the initial data $(f,g)$ (see (\ref{data1})) satisfies 
\begin{equation}\label{dontforgetsmall}
\|f\|_{H^3}+\|g\|_{H^2}<{\hat\varepsilon}
\end{equation} 
so that for all $i=1,\dots,N$ and $\alpha,\beta=0,1,2,3$ 
the quantities 
$|F_i^{j,\alpha\beta 0}g_j
+
F_i^{j,\alpha\beta k}\partial_k f_j|$ and 
$|G_i^{\alpha\beta}(f,g,\partial_x f)|$ 
(see (\ref{nlt1})--(\ref{nlt2})) 
are also small by the Sobolev embedding 
and we can therefore rely upon 
the local existence theorem mentioned above. 

We may focus on a priori estimates for the local solutions. 
Let us first consider the case where initial data are smooth and compactly supported 
so that there exists a constant $R>0$ such that 
${\rm supp}\,\{f,g\}\subset \{x\in{\mathbb R}^3:|x|<R\}$. 
Moreover, we know that the local solution is smooth and satisfies 
\begin{equation}\label{fps}
u_i(t,x)=\partial_t u_i(t,x)=0,\quad 
|x|>t+R,\,\,i=1,\dots,N.
\end{equation}
John proved (\ref{fps}) for $C^2$ solutions for scalar wave equations 
(see \cite{John1981}, \cite{John1990}), 
and his proof is obviously valid for the systems (\ref{eq1}) under consideration. 
We temporarily assume the regularity and support conditions on the data 
because the proof of Theorem \ref{maintheorem} becomes easier. 
Note, however, that 
all the constants $C$ appearing below 
will never depend upon 
this constant $R$, and these conditions can be finally removed 
by the standard argument. 

Recall the definition of $D(f,g)$ (see (\ref{small7})). 
To prove the global existence, 
we must assume that the initial data $(f,g)$ is smaller than 
we have done in (\ref{dontforgetsmall}). 
That is, 
using some appropriate constants which will appear later 
in our discussion, we assume 
\begin{align}\label{small31}
D&(f,g)\leq\varepsilon_0\\
&
\mbox{where}\,
0<\varepsilon_0
<
\min
\biggl\{
1,
{\hat\varepsilon},\,
\frac{\varepsilon_1^*}{6C^*C_0},\,
\frac{1}{2C_0(3C_{12}+4C_0C_{14})},\,
\frac{C_{11}}{C_{15}},\,
\frac{C_{21}}{2C_{22}},\,
\frac{C_{31}}{2C_{32}}
\biggr\}.\nonumber
\end{align}
Using the equality 
$\partial_t L_j v=\partial_j v+x_j\partial_t^2 v$ at $t=0$, 
the equation (\ref{eq1}), and the Sobolev type inequality (\ref{sob25}), 
we see that there exists a numerical constant 
$C_d>0$ such that 
the local solution initially satisfies 
\begin{equation}\label{small32}
N_4(u(0)),\,M_3(u(0)),\,X_2(u(0))
\leq
C_d D(f,g),
\end{equation}
due to the size condition (\ref{small31}). 
We remark that 
the equality 
$$
r\partial_k
=
\omega_k\Lambda
+
\sum_{j \ne k}\omega_j\Omega_{jk}
$$ 
plays a role in showing (\ref{small32}).
(Here, we have used the notation $\Omega_{ij}:=x_i\partial_j-x_j\partial_i$ 
not only for $i<j$ but also for $i>j$.)

Before starting a priori estimates for the local solutions, 
we must mention some point-wise estimates as in \cite{H2016}, 
\cite{HZ2019}. 
These inequalities compensate for the absence of 
$\partial_t^i w(t,x)$ $(i = 2,\,3,\,4)$ 
in the definition of the norms (\ref{e1}), (\ref{n4}), 
(\ref{mj+1}), (\ref{xj}), 
(\ref{gvt404})--(\ref{lv311322}). 
For the local solutions $u$, we use the notation 
\begin{align}\label{langlerangle}
&\langle\!\langle u(t)\rangle\!\rangle
:=
\sum_{i=1}^N
\biggl\{
\sum_{|a|+|b|\leq 1}
\langle t\rangle 
\|
\partial\partial_x^a\Omega^b u_i(t)
\|_{L^\infty({\mathbb R}^3)}\\
&
\hspace{2.7cm}
+
\sum_{|c|=1}
\|\partial L^c u_i(t)\|_{L^\infty({\mathbb R}^3)}
+
\|\partial Su_i(t)\|_{L^\infty({\mathbb R}^3)}\nonumber\\
&
\hspace{2.7cm}
+
\sum_{|b|\leq 1}
\langle t\rangle^{1-\delta}
\|\Omega^b u_i(t)\|_{L^\infty({\mathbb R}^3)}
+
\sum_{|b|=2}
\|\Omega^b u_i(t)\|_{L^\infty({\mathbb R}^3)}
\nonumber\\
&
\hspace{2.7cm}
+
\sum_{|b|\leq 1}\sum_{|c|+d=1}
\langle t\rangle^{-\delta}
\|\Omega^b L^c S^d u_i(t)\|_{L^\infty({\mathbb R}^3)}\biggr\}.\nonumber
\end{align}
Recall the definition of the notation $Z_k$. 

\begin{lemma}\label{pointwise}
There exists a small constant $\varepsilon_1^*>0$ with the following 
property\,$:$ 
whenever smooth solutions to $(\ref{eq1})$ satisfy 
\begin{equation}\label{lem311}
\langle\!\langle u(t)\rangle\!\rangle
\leq
\varepsilon_1^*,
\end{equation}
the point-wise inequalities hold 
for $i=1,2,\dots,N$, 
$j=1,2,3$, 
$k=1,2,\dots,10$, and 
$l=1,2,\dots,6:$
\begin{equation}\label{lem312}
|\partial_t^2 u_i(t,x)|
\leq
C\sum_{|a|\leq 1}|\partial\partial_x^a u(t,x)|
+
C|u(t,x)|^3,
\end{equation}
\begin{equation}\label{lem3125}
|\partial_t^3 u_i(t,x)|
\leq
C\sum_{|a|\leq 2}|\partial\partial_x^a u(t,x)|
+
C|u(t,x)|^3,
\end{equation}
\begin{align}\label{lem313}
|\partial_t^2Z_k u_i(t,x)|
\leq&
C\sum_{|a|\leq 1}
(|\partial\partial_x^aZ_k u(t,x)|
+
|\partial\partial_x^a u(t,x)|)\\
&
+
C|u(t,x)|^2|Z_k u(t,x)|
+
C|u(t,x)|^3,\nonumber
\end{align}
\begin{align}\label{lem314}
|&\partial_t^2 Z_l Z_k u_i(t,x)|\\
&
\leq
C\sum_{|a|\leq 1}
(
|\partial\partial_x^a Z_l Z_k u(t,x)|
+
|\partial\partial_x^a Z_l u(t,x)|
+
|\partial\partial_x^a Z_k u(t,x)|
+
|\partial\partial_x^a u(t,x)|)\nonumber\\
&
+
C|u(t,x)||Z_l u(t,x)||Z_k u(t,x)|
+
C|u(t,x)|^2|Z_k Z_l u(t,x)|\nonumber\\
&
+
C|u(t,x)|^2|Z_l u(t,x)|
+
C|u(t,x)|^2|Z_k u(t,x)|
+
C|u(t,x)|^3,\nonumber
\end{align}
\begin{align}\label{lem3145}
|\partial_t^3 Z_k u_i(t,x)|
\leq&
C\sum_{|a|\leq 2}|\partial\partial_x^a Z_k u(t,x)|
+
C\sum_{|a|\leq 2}|\partial\partial_x^a u(t,x)|\\
&
+
C|u(t,x)|^2|Z_k u(t,x)|
+
C|u(t,x)|^3,\nonumber
\end{align}
\begin{equation}\label{lem315}
|T_j\partial_t^2 u_i(t,x)|
\leq
C|T_j\partial\partial_x u(t,x)|
+
C|T_j\partial u(t,x)|
+
C|T_j u(t,x)|,
\end{equation}
and
\begin{align}\label{lem316}
|&T_j\partial_t^2 Z_k u_i(t,x)|\\
&
\leq
C|T_j\partial\partial_x Z_k u(t,x)|\nonumber\\
&
+
C(|T_j\partial u(t,x)|+|T_j\partial Z_k u(t,x)|)
+
C(|T_j u(t,x)|+|T_j Z_k u(t,x)|)\nonumber\\
&
+
C|T_j\partial u(t,x)||\partial\partial_x Z_k u(t,x)|
+
C|T_j u(t,x)||\partial\partial_x Z_k u(t,x)|\nonumber\\
&
+
C|Z_k u(t,x)|
(|T_ju(t,x)|+|T_j\partial u(t,x)|
+|T_j\partial\partial_x u(t,x)|).\nonumber
\end{align}
\end{lemma}
The proof is based on straightforward computations. 
Note that we have not pursued the best possible. 
The above inequalities suffice for our purpose.

We may obviously focus on the energy of the highest order. 
Moreover, we may focus on the bound for 
$E_1({\bar Z}^a Su_i(t))$ $(i=1,\dots,N,\,|a|=2)$ 
because we can obtain a similar bound 
for $E_1({\bar Z}^a L^c u_i(t))$ 
$(|a|=2,\,|c|=1)$ in the same way 
and the bound for $E_1({\bar Z}^a u_i(t))$ $(|a|=3)$ 
is easier to get. 

Using Lemma \ref{lemmacomm}, 
we get $\square{\bar Z}^a Su_i
={\bar Z}^a S \square u_i + 2{\bar Z}^a \square u_i$. 
By Lemma \ref{nullpreserved}, 
we therefore obtain for $|a|=2$ 
\begin{align}\label{dec1}
\Box&{\bar Z}^a Su_i\\
&
+
F_i^{j,\alpha\beta\gamma}
(\partial_\gamma u_j)
(\partial_{\alpha\beta}^2 {\bar Z}^a Su_i)
+
G_i^{\alpha\beta}(u,\partial u)
(\partial_{\alpha\beta}^2 {\bar Z}^a Su_i)\nonumber\\
&
+
\sum\!{}^{'}
{\hat F}_i^{j,\alpha\beta\gamma}
(\partial_\gamma{\bar Z}^{a'}S^{d'}u_j)
(\partial_{\alpha\beta}^2 {\bar Z}^{a''} S^{d''}u_i)\nonumber\\
&
+
\sum\!{}^{''}
{\hat F}_i^{jk,\alpha\beta}
(\partial_\alpha{\bar Z}^{a'}S^{d'}u_j)
(\partial_\beta{\bar Z}^{a''} S^{d''}u_k)\nonumber\\
&
+
\biggl(
{\bar Z}^a S
\bigl(
G_i^{\alpha\beta}(u,\partial u)\partial_{\alpha\beta}^2 u_i
\bigr)
-
G_i^{\alpha\beta}(u,\partial u)
\partial_{\alpha\beta}^2{\bar Z}^a Su_i
\biggr)\nonumber\\
&
+
{\bar Z}^a SH_i(u,\partial u)
-2{\bar Z}^a C_i(u,\partial u,\partial^2 u)=0.\nonumber
\end{align}
Here, for given $i,\,j$ 
the new coefficients ${\hat F}_i^{j,\alpha\beta\gamma}$ 
satisfy the null condition. 
Also, for given $i,\,j$, and $k$ 
the new coefficients ${\hat F}_i^{jk,\alpha\beta}$ 
satisfy the null condition. 
By $\sum\!{}^{'}$ 
we mean the summation over 
$|a'|+|a''|\leq 2,\,d'+d''\leq 1$, $|a''|+d''\leq 2$. 
Also, by $\sum\!{}^{''}$ 
we mean the summation over 
$|a'|+|a''|\leq 2$, $d'+d''\leq 1$. 
Just for simplicity of notation, 
we have omitted the dependence of 
${\hat F}_i^{j,\alpha\beta\gamma}$ 
and ${\hat F}_i^{jk,\alpha\beta}$ 
on $a,\,a',\,a'',\,d,\,d'$, and $d''$. 
Using Lemma \ref{energymomentum}, (\ref{2019june26t00}) 
and (\ref{2019june26t0beta}) 
for $v={\bar Z}^aSu_i$, 
$h^{\alpha\beta}=F_i^{j,\alpha\beta\gamma}\partial_\gamma u_j
+
G_i^{\alpha\beta}(u,\partial u)$, 
we obtain 
for every $i=1,\dots,N$ and the function $g=g(t-r)$ chosen below 
(see (\ref{howg})) 
\begin{align}\label{dec2}
&
\frac12
\partial_t
\bigl\{
e^g
\bigl(
(\partial_t{\bar Z}^{a}S u_i)^2
+
|\nabla{\bar Z}^{a}S u_i|^2
\\
&
\hspace{2cm}
+
2F_i^{j,0\beta\gamma}
(\partial_\gamma u_j)
(\partial_\beta{\bar Z}^{a}S u_i)
(\partial_t{\bar Z}^{a}S u_i)\nonumber\\
&
\hspace{2cm}
+
2G_i^{0\beta}(u,\partial u)
(\partial_\beta{\bar Z}^{a}S u_i)
(\partial_t{\bar Z}^{a}S u_i)\nonumber\\
&
\hspace{2cm}
-
F_i^{j,\alpha\beta\gamma}
(\partial_\gamma u_j)
(\partial_\beta{\bar Z}^{a}S u_i)
(\partial_\alpha{\bar Z}^{a}S u_i)\nonumber\\
&
\hspace{2cm}
-
G_i^{\alpha\beta}(u,\partial u)
(\partial_\beta{\bar Z}^{a}S u_i)
(\partial_\alpha{\bar Z}^{a}S u_i)
\bigr)
\bigr\}\nonumber\\
&
+\nabla\cdot\{\cdots\}
+
e^gq
+
e^g(J_{i,1}+J_{i,2}+J_{i,3}+J_{i,4}+J_{i,5})=0,\nonumber
\end{align}
Here, 
\begin{equation}\label{feb19}
q
=q_1
-\frac12g'(t-r)\sum_{j=1}^3
(T_j{\bar Z}^{a}S u_i)^2
-g'(t-r)q_2,
\end{equation}
\begin{align}\label{dec3}
q_1
=&
-
F_i^{j,\alpha\beta\gamma}
(\partial_{\alpha\gamma}^2 u_j)
(\partial_\beta{\bar Z}^{a}S u_i)
(\partial_t{\bar Z}^{a}S u_i)\\
&
-
\bigl(\partial_\alpha G_i^{\alpha\beta}(u,\partial u)\bigr)
(\partial_\beta{\bar Z}^{a}S u_i)
(\partial_t{\bar Z}^{a}S u_i)\nonumber\\
&
+
\frac12
F_i^{j,\alpha\beta\gamma}
(\partial_{t\gamma}^2 u_j)
(\partial_\beta{\bar Z}^{a}S u_i)
(\partial_\alpha{\bar Z}^{a}S u_i)\nonumber\\
&
+
\frac12
\bigl(\partial_t G_i^{\alpha\beta}(u,\partial u)\bigr)
(\partial_\beta{\bar Z}^{a}S u_i)
(\partial_\alpha{\bar Z}^{a}S u_i),\nonumber
\end{align}
\begin{align}\label{dec4}
q_2
=
&
-
\omega_\alpha
F_i^{j,\alpha\beta\gamma}
(\partial_{\gamma} u_j)
(\partial_\beta{\bar Z}^{a}S u_i)
(\partial_t{\bar Z}^{a}S u_i)\\
&
-
\omega_\alpha
G_i^{\alpha\beta}(u,\partial u)
(\partial_\beta{\bar Z}^{a}S u_i)
(\partial_t{\bar Z}^{a}S u_i)\nonumber\\
&
-
\frac12
F_i^{j,\alpha\beta\gamma}
(\partial_{\gamma} u_j)
(\partial_\beta{\bar Z}^{a}S u_i)
(\partial_\alpha{\bar Z}^{a}S u_i)\nonumber\\
&
-
\frac12
G_i^{\alpha\beta}(u,\partial u)
(\partial_\beta{\bar Z}^{a}S u_i)
(\partial_\alpha{\bar Z}^{a}S u_i),\nonumber
\end{align}
where $\omega_0=-1$, $\omega_k=x_k/|x|$, $k=1,2,3$. 
Also, (see (\ref{dec1}) above for $\sum\!{}^{'}$, $\sum\!{}^{''}$)
\begin{align}
&
J_{i,1}
=
\sum\!{}^{'}
{\hat F}_i^{j,\alpha\beta\gamma}
(\partial_\gamma{\bar Z}^{a'}S^{d'}u_j)
(\partial_{\alpha\beta}^2 {\bar Z}^{a''} S^{d''}u_i)
(\partial_t{\bar Z}^aS u_i),\label{dec5}\\
&
J_{i,2}
=
\sum\!{}^{''}
{\hat F}_i^{jk,\alpha\beta}
(\partial_\alpha{\bar Z}^{a'}S^{d'}u_j)
(\partial_\beta{\bar Z}^{a''} S^{d''}u_k)
(\partial_t{\bar Z}^aS u_i),\label{dec6}\\
&
J_{i,3}
=
\biggl(
{\bar Z}^a S
\bigl(
G_i^{\alpha\beta}(u,\partial u)\partial_{\alpha\beta}^2 u_i
\bigr)
-
G_i^{\alpha\beta}(u,\partial u)
\partial_{\alpha\beta}^2{\bar Z}^a Su_i
\biggr)
(\partial_t{\bar Z}^aS u_i),\label{dec7}\\
&
J_{i,4}
=
\bigl(
{\bar Z}^a SH_i(u,\partial u)
\bigr)
(\partial_t{\bar Z}^aS u_i),\label{dec8}\\
&
J_{i,5}\label{july6}
=
\bigl(
-2{\bar Z}^a C_i(u,\partial u,\partial^2 u)
\bigr)
(\partial_t{\bar Z}^a Su_i).
\end{align}
As in \cite{HZ2019}, we use the following quantities 
$G(v(t))$ and $L(v(t))$ 
which are related to the ghost energy 
and the localized energy, respectively\,:
\begin{align}\label{gvt404}
G(v(t))
:=
\biggl\{
\sum_{j=1}^3
\biggl(&
\sum_{{|a|+|c|+d\leq 3}\atop{|c|+d\leq 1}}
\|
\langle t-r\rangle^{-(1/2)-\eta}
T_j{\bar Z}^a L^c S^d v(t)
\|_{L^2({\mathbb R}^3)}^2\\
&
+
\sum_{{|a|+|c|+d\leq 2}\atop{|c|+d\leq 1}}
\|
\langle t-r\rangle^{-(1/2)-\eta}
T_j\partial_t{\bar Z}^a L^c S^d v(t)
\|_{L^2({\mathbb R}^3)}^2
\biggr)
\biggr\}^{1/2},\nonumber
\end{align}
\begin{align}\label{lv311322}
L(v(t)):=
\biggl\{
\sum_{{|a|+|c|+d\leq 3}\atop{|c|+d\leq 1}}
\biggl(&
\|
r^{-5/4}
{\bar Z}^a L^c S^d v(t)
\|_{L^2({\mathbb R}^3)}^2\\
&
+
\|
r^{-1/4}
\partial{\bar Z}^a L^c S^d v(t)
\|_{L^2({\mathbb R}^3)}^2
\biggr)
\biggr\}^{1/2}.\nonumber
\end{align}
We remark that the norm 
$\|
\langle t-r\rangle^{-(1/2)-\eta}
T_j\partial_t{\bar Z}^a L^c S^d v(t)
\|_{L^2({\mathbb R}^3)}$ 
$(|a|+|c|+d\leq 2,\,|c|+d\leq 1)$, 
which requires a separate treatment, 
naturally comes up later. 
See, e.g., (\ref{troublesome}) below. 
For $w(t,x)=(w_1(t,x),\dots,w_N(t,x))$, we set 
\begin{equation}
G(w(t))
:=
\biggl(
\sum_{i=1}^N
G(w_i(t))^2
\biggr)^{1/2},\quad
L(w(t))
:=
\biggl(
\sum_{i=1}^N
L(w_i(t))^2
\biggr)^{1/2}.
\end{equation}
Recall the definition of $D(f,g)$ (see (\ref{small7})). 
By $\eta$, we mean a sufficiently small positive constant 
satisfying $0<\eta<1/3$. 
The main purpose of this section is to prove\,:
\begin{proposition}\label{ghostenergyestimate}
Suppose that initial data $(\ref{data1})$ is smooth and compactly supported, 
and suppose that the local solution $u$ satisfies $(\ref{lem311})$ in 
some interval $(0,T)$. 
Then the following inequality holds 
for all $t\in (0,T):$
\begin{align}\label{genergyest}
N&_4(u(t))^2
+
\int_0^t
G(u(\tau))^2
d\tau\\
&
\leq
CN_4(u(0))^2+CD(f,g)^6\nonumber\\
&
\hspace{0.2cm}
+
C\int_0^t
\langle\tau\rangle^{-1}
\bigl(
M_3(u(\tau))+N_4(u(\tau))
\bigr)
L(u(\tau))^2 d\tau\nonumber\\
&
\hspace{0.2cm}
+
C\int_0^t
\langle \tau\rangle^{-3/2}
\bigl(
M_3(u(\tau))
+
N_4(u(\tau))
\bigr)
N_4(u(\tau))^2 d\tau\nonumber\\
&
\hspace{0.2cm}
+
C\int_0^t
\langle \tau\rangle^{-1+\eta}
\bigl(
M_3(u(\tau))
+
N_4(u(\tau))
\bigr)
N_4(u(\tau))G(u(\tau)) d\tau\nonumber\\
&
\hspace{0.2cm}
+
C\int_0^t
\langle \tau\rangle^{-1}
\bigl(
M_3(u(\tau))^2
+
N_4(u(\tau))^2
\bigr)
L(u(\tau))^2d\tau\nonumber\\
&
\hspace{0.2cm}
+
C\int_0^t
\langle \tau\rangle^{-2}
\bigl(
M_3(u(\tau))^3
+
N_4(u(\tau))^3
\bigr)
N_4(u(\tau))d\tau\nonumber\\
&
\hspace{0.2cm}
+
C\int_0^t
\langle \tau\rangle^{-2}
\bigl(
M_3(u(\tau))^2
+
N_4(u(\tau))^2
\bigr)
N_4(u(\tau))^2 d\tau.\nonumber
\end{align}
\end{proposition}
\begin{proof}
Due to (\ref{lem311}), we may use Lemma \ref{pointwise} repeatedly. 
We also remark that in view of (\ref{n4}) and (\ref{mj+1}), 
we have the Sobolev-type inequalities 
\begin{align}
&
\|\langle t-r\rangle v(t)\|_{L^6({\mathbb R}^3)}
\leq
CM_1(v(t))+CN_1(v(t)),\label{3ineq1}\\
&
\|\langle t-r\rangle \partial v(t)\|_{L^6({\mathbb R}^3)}
\leq
CN_2(v(t)),\label{3ineq2}\\
&
\langle t-r\rangle |v(t,x)|
\leq
CM_1(v(t))+CN_2(v(t)),\label{3ineq3}\\
&
\langle t-r\rangle |\partial v(t,x)|
\leq
CN_3(v(t)),\label{3ineq4}\\
&
\|r^{(1/2)+\theta}\langle t-r\rangle^{1-\theta}v(t,x)\|_{L_r^\infty L_\omega^4}
\leq
CM_1(u(t))+CN_1(u(t)),
\label{hypart1}\\
&
\|r^{(1/2)+\theta}\langle t-r\rangle^{1-\theta}\partial v(t,x)\|_{L_r^\infty L_\omega^4}
\leq
CN_2(u(t))
\label{hyderivative}
\end{align}
(see Lemma \ref{manysob} and Lemma \ref{hyinequality}) which will be frequently employed 
in the following discussion.

The estimate of the $L^1({\mathbb R}^3)$-norm 
of each term in (\ref{dec2}) is carried out 
over the set $\{x\in{\mathbb R}^3:|x|<(1+t)/2\}$ 
and its complement set, separately, for any fixed time $t\in (0,T)$. 
It is therefore useful to introduce 
the characteristic function 
$\chi_1(x)$ of the former set, 
and we set $\chi_2(x):=1-\chi_1(x)$. 

\subsection{Estimate over the set $\{x\in{\mathbb R}^3:|x|<(1+t)/2\}$}~\\
\noindent{\bf $\cdot$Estimate of $\chi_1 q$.} 
Recall the definition of $q$, $q_1$, and $q_2$ 
(see (\ref{feb19}), (\ref{dec3}), and (\ref{dec4})). 
Due to (\ref{lem311}), we easily obtain the elementary bound
\begin{equation}\label{bdd22}
|G_i^{\alpha\beta}(u,\partial u)|
\leq
C(|u|+|\partial u|),
\quad
|\partial G_i^{\alpha\beta}(u,\partial u)|
\leq
C(|\partial u|+|\partial^2 u|).
\end{equation}
Using (\ref{lem312}) to handle $\partial_t^2 u_j(t,x)$ 
and using (\ref{lem311}) to get the simple inequality 
$|u(t,x)|^3\leq C|u(t,x)|$, we obtain 
\begin{align}\label{1817}
\|&\chi_1q_1\|_{L^1({\mathbb R}^3)}\\
&
\leq
C\sum_{|b|\leq 1}
\|
\chi_1
|\partial\partial_x^b u(t)|
|\partial {\bar Z}^a S u_i(t)|^2
\|_{L^1({\mathbb R}^3)}
+
C
\|
\chi_1
|u(t)|
|\partial {\bar Z}^a S u_i(t)|^2
\|_{L^1({\mathbb R}^3)}\nonumber\\
&
\leq
C
\langle t\rangle^{-1}
\biggl(
\sum_{|b|\leq 1}
\|
r^{1/2}\langle t-r\rangle
\partial\partial_x^b u(t)
\|_{L^\infty({\mathbb R}^3)}
\biggr)
\|
r^{-1/4}\partial{\bar Z}^a S u_i(t)
\|_{L^2({\mathbb R}^3)}^2\nonumber\\
&
+
C
\langle t\rangle^{-1}
\|r^{1/2}\langle t-r\rangle u(t)\|_{L^\infty({\mathbb R}^3)}
\|
r^{-1/4}
\partial {\bar Z}^a S u_i(t)
\|_{L^2({\mathbb R}^3)}^2\nonumber\\
&
\leq
C\langle t\rangle^{-1}
\bigl(
M_2(u(t))
+
N_4(u(t))
\bigr)
L(u(t))^2,\nonumber
\end{align}
where we have used (\ref{hyderivative}) and (\ref{hypart1}) with $\theta=0$ 
together with the Sobolev embedding 
$W^{1,4}(S^2)\hookrightarrow L^\infty(S^2)$. 

For the estimate of $\chi_1 g'(t-r)q_2$, 
we choose $g=g(\rho)$ 
$(\rho\in{\mathbb R})$ 
\begin{equation}\label{howg}
g(\rho)
=
-\int_0^\rho
\langle
{\hat \rho}
\rangle^{-1-2\eta}
d{\hat \rho}
\quad
\mbox{so that}\,\,\,
g'(\rho)
=
-\langle \rho\rangle^{-1-2\eta}.
\end{equation}
Using the first inequality in (\ref{bdd22}) and 
(\ref{3ineq3})--(\ref{3ineq4}), we get
\begin{align}
\|&
\chi_1g'(t-r)q_2
\|_{L^1({\mathbb R}^3)}\\
&
\leq
C\langle t\rangle^{-2-2\eta}
\bigl(
\|
\chi_1\langle t-r\rangle u(t)
\|_{L^\infty({\mathbb R}^3)}
+
\|
\chi_1\langle t-r\rangle \partial u(t)
\|_{L^\infty({\mathbb R}^3)}
\bigr)\nonumber\\
&
\hspace{2.3cm}
\times
\|
\partial{\bar Z}^a Su_i(t)
\|_{L^2({\mathbb R}^3)}^2\nonumber\\
&
\leq
C\langle t\rangle^{-2-2\eta}
\bigl(
M_1(u(t))
+
N_3(u(t))
\bigr)
N_4(u(t))^2.
\nonumber
\end{align}
The estimate of $\chi_1 q$ has been finished. 

\noindent{\bf$\cdot$Estimate of $\chi_1 J_{i,1}$.} 
We next estimate $\chi_1 J_{i,1}$ 
by basically following \cite{HZ2019} and 
paying attention on the number of 
occurrences of $S$. 

\noindent \underline{Case 1. $d'=1$, $d''=0$.} 

\noindent \underline{Case 1-1. $|a'|=0$, $|a''|\leq 2$.} 
We employ (\ref{lem312}), (\ref{lem313}) and (\ref{lem314}) 
to deal with $\partial_t^2{\bar Z}^{a''}u_i$, 
and we then use (\ref{lem311}) to get 
$|u||{\bar Z}^{b} u|^2$, 
$|u|^2|{\bar Z}^{b'} u|\leq C|u|$ 
for $|b|\leq 1$ and $|b'|\leq 2$. 
In this way, we get
\begin{align}\label{2019aug31456}
\|&
\chi_1(\partial Su_j(t))
(\partial^2{\bar Z}^{a''}u_i(t))
(\partial_t{\bar Z}^a Su_i(t))\|_{L^1({\mathbb R}^3)}\\
&
\leq
C\sum_{|b|\leq |a''|}
\|
\chi_1
(\partial Su_j(t))
(\partial\partial_x{\bar Z}^b u(t))
(\partial_t{\bar Z}^a Su_i(t))\|_{L^1({\mathbb R}^3)}\nonumber\\
&
\hspace{0.2cm}
+
C\sum_{|b|\leq |a''|}
\|
\chi_1
(\partial Su_j(t))
(\partial{\bar Z}^b u(t))
(\partial_t{\bar Z}^a Su_i(t))\|_{L^1({\mathbb R}^3)}\nonumber\\
&
\hspace{0.2cm}
+
C
\|
\chi_1
(\partial Su_j(t))
|u(t)|
(\partial_t{\bar Z}^a Su_i(t))\|_{L^1({\mathbb R}^3)}
\nonumber\\
&
=:
J_{i,1}^{(1)}
+
J_{i,1}^{(2)}
+
J_{i,1}^{(3)}.\nonumber
\end{align}
Using (\ref{sob23}) and (\ref{trd}), we obtain 
\begin{align}\label{estji11}
J_{i,1}^{(1)}
&
\leq
C\sum_{|b|\leq 2}
\langle t\rangle^{-1}
\| r^{1/2}\partial S u_j(t)\|_{L^\infty({\mathbb R}^3)}
\|r^{-1/4}\langle t-r\rangle\partial_x\partial{\bar Z}^b u(t)\|_{L^2({\mathbb R}^3)}\\
&
\hspace{2.3cm}
\times
\| r^{-1/4}\partial_t{\bar Z}^a S u_i(t)\|_{L^2({\mathbb R}^3)}\nonumber\\
&
\leq
C\langle t\rangle^{-1}N_4(u(t))L(u(t))^2.\nonumber
\end{align}
We use (\ref{hyderivative}) with $\theta=0$ to get
\begin{align}\label{estj12}
J_{i,1}^{(2)}&\leq
C\sum_{|b|\leq 2}
\langle t\rangle^{-1}
\|r^{-1/4}\partial S u_j(t)\|_{L^2_r L^4_\omega}
\|r^{1/2}\langle t-r\rangle\partial{\bar Z}^b u(t)\|_{L^\infty_r L^4_\omega}\\
&
\hspace{2cm}
\times
\|r^{-1/4}\partial_t{\bar Z}^a S u_i(t)\|_{L^2({\mathbb R}^3)}\nonumber\\
&
\leq
C\langle t\rangle^{-1}
N_4(u(t))L(u(t))^2.\nonumber
\end{align}
Similarly, we obtain by (\ref{hypart1})
\begin{equation}\label{estj13}
J_{i,1}^{(3)}
\leq
C
\langle t\rangle^{-1}
\bigl(
M_1(u(t))
+
N_1(u(t))
\bigr)
L(u(t))^2.
\end{equation}
\noindent\underline{Case 1-2. $|a'|\leq 1$, $|a''|\leq 1$.} 
Using (\ref{lem311}), (\ref{lem312}), and (\ref{lem313}), we get
\begin{align}
\|&
\chi_1
(\partial{\bar Z}^{a'}Su_j(t))
(\partial^2{\bar Z}^{a''}u_i(t))
(\partial_t{\bar Z}^a Su_i(t))\|_{L^1({\mathbb R}^3)}\\
&
\leq
C\sum_{|b|\leq |a''|}\sum_{|b'|\leq 1}
\|
\chi_1
(\partial{\bar Z}^{a'} Su_j(t))
(\partial\partial_x^{b'}{\bar Z}^b u(t))
(\partial_t{\bar Z}^a Su_i(t))\|_{L^1({\mathbb R}^3)}
\nonumber\\
&
\hspace{0.2cm}
+
C
\|
\chi_1
(\partial{\bar Z}^{a'} Su_j(t))
|u(t)|
(\partial_t{\bar Z}^a Su_i(t))\|_{L^1({\mathbb R}^3)}
\nonumber\\
&
=:
J_{i,1}^{(4)}+J_{i,1}^{(5)}.\nonumber
\end{align}
We have only to handle $J_{i,1}^{(4)}$ and $J_{i,1}^{(5)}$ 
in the same way as in (\ref{estj12}) and (\ref{estj13}), 
respectively. 

\noindent\underline{Case 1-3. $|a'|\leq 2$, $|a''|=0$.} 
Proceeding as in (\ref{1817}), we can obtain 
\begin{align}\label{estjan25}
\|&
\chi_1
(\partial{\bar Z}^{a'}Su_j(t))
(\partial^2 u_i(t))
(\partial_t{\bar Z}^a Su_i(t))\|_{L^1({\mathbb R}^3)}\\
&
\leq
C
\langle t\rangle^{-1}
\bigl(
M_2(u(t))
+
N_4(u(t))
\bigr)
L(u(t))^2.\nonumber
\end{align}
\noindent\underline{Case 2.} $d'=0$, $d''=1$. In this case, 
we know $|a''|\leq 1$. 

\noindent\underline{Case 2-1. $|a'|\leq 1$, $|a''|\leq 1$.} 

Using (\ref{lem311}), (\ref{lem313}) and (\ref{lem314}), we get
\begin{align}
\|&
\chi_1
(\partial{\bar Z}^{a'} u_j(t))
(\partial^2{\bar Z}^{a''} S u_i(t))
(\partial_t{\bar Z}^a Su_i(t))\|_{L^1({\mathbb R}^3)}\\
&
\leq
C
\sum_{|b|, |b'|\leq 1}
\|
\chi_1
(\partial{\bar Z}^{a'} u_j(t))
(\partial\partial_x^{b'}{\bar Z}^b S u(t))
(\partial_t{\bar Z}^a Su_i(t))\|_{L^1({\mathbb R}^3)}
\nonumber\\
&
\hspace{0.2cm}
+
C
\sum_{|b|\leq 1}
\|
\chi_1
(\partial{\bar Z}^{a'} u_j(t))
|u(t)|
(\partial_t{\bar Z}^a Su_i(t))\|_{L^1({\mathbb R}^3)}
\nonumber\\
&
=:
J_{i,1}^{(6)}+J_{i,1}^{(7)}.\nonumber
\end{align}
Here, the term $J_{i,1}^{(7)}$ has appeared 
because we have used (\ref{lem311}) to get 
\begin{align}\label{product320}
|&u(t,x)||{\bar Z}^b u(t,x)||S u(t,x)|,\,\,
|u(t,x)|^2|{\bar Z}^b S u(t,x)|\\
&
\leq
C
\langle t\rangle^{-1+2\delta}(\varepsilon_1^*)^2|u(t,x)|
\leq
C|u(t,x)|,\,\,\,|b|\leq 1.\nonumber
\end{align}
(Recall that we are assuming $\delta\leq 1/2$.) 
In order to bound 
$\|r^{1/2}\langle t-r\rangle\partial{\bar Z}^{a'} u_j(t)\|_{L^\infty({\mathbb R}^3)}$ 
and 
$\|r^{1/2}\langle t-r\rangle\partial u(t)\|_{L^\infty({\mathbb R}^3)}$, 
we employ (\ref{hyderivative}) and (\ref{hypart1}) with $\theta=0$ 
together with the Sobolev embedding 
$W^{1,4}(S^2)\hookrightarrow L^\infty(S^2)$, as in (\ref{1817}). 
We thus get 
\begin{equation}\label{ji16march20}
J_{i,1}^{(6)}+J_{i,1}^{(7)}
\leq
C\langle t\rangle^{-1}
\bigl(
M_2(u(t))
+
N_4(u(t))
\bigr)
L(u(t))^2.
\end{equation}

\noindent\underline{Case 2-2. $|a'|\leq 2$, $|a''|=0$.} 
We suitably modify the argument in (\ref{ji16march20}) 
by employing the $L^\infty_r L^4_\omega$ and the $L^2_r L^4_\omega$ norms. 
Using (\ref{lem311}) and (\ref{lem313}), we get 
\begin{align} 
\|&
\chi_1
(\partial{\bar Z}^{a'} u_j(t))
(\partial^2 S u_i(t))
(\partial_t{\bar Z}^a Su_i(t))\|_{L^1({\mathbb R}^3)}\\
&
\leq
C\langle t\rangle^{-1}
\bigl(
M_2(u(t))
+
N_4(u(t))
\bigr)
L(u(t))^2.\nonumber
\end{align}
\noindent{\bf $\cdot$Estimate of $\chi_1 J_{i,2}$.} 
It suffices to explain how to bound 
$$
\|
\chi_1
(\partial{\bar Z}^{a'} u_j(t))
(\partial{\bar Z}^{a''} S u_k(t))
(\partial_t{\bar Z}^a Su_i(t))\|_{L^1({\mathbb R}^3)}
$$ 
for $|a'|+|a''|\leq 2$. We employ 
$\|
\chi_1
r^{1/2}\langle t-r\rangle
\partial{\bar Z}^{a'} u_j(t)\|_{L^\infty({\mathbb R}^3)}$ 
if $|a'|\leq 1$, 
$\|
\chi_1
r^{1/2}\langle t-r\rangle
\partial{\bar Z}^{a'} u_j(t)\|_{L_r^\infty L_\omega^4({\mathbb R}^3)}$ 
if $|a'|=2$. 
We then obtain 
\begin{equation}
\|\chi_1 J_{i,2}\|_{L^1({\mathbb R}^3)}
\leq
C\langle t\rangle^{-1}
N_4(u(t))L(u(t))^2.
\end{equation}
\noindent{\bf $\cdot$Estimate of $\chi_1 J_{i,3}$.} 
We recall that 
$G_i^{\alpha\beta}(u,v)$ is a 
homogeneous polynomial of degree $2$, 
and therefore 
$G_i^{\alpha\beta}(u,\partial u)$ has the form of 
sum of constant multiples of 
$u_ju_k$, 
$u_j\partial_\gamma u_k$, 
and $(\partial_\gamma u_j)(\partial_\delta u_k)$. 
For the estimate of $\chi_1 J_{i,3}$, 
it suffices to repeat the same argument as we have done above 
and obtain 
\begin{align}\label{boundj13}
\|\chi_1 J_{i,3}\|_{L^1({\mathbb R}^3)}
\leq&
C
\langle t\rangle^{-1}
\bigl(
M_3(u(t))^2
+
N_4(u(t))^2
\bigr)
L(u(t))^2
\\
&
+
C
\langle t\rangle^{-2}
\bigl(
M_2(u(t))
+
N_3(u(t))
\bigr)^3
N_4(u(t)).\nonumber
\end{align}
As for the proof of this bound, 
it suffices to mention how to deal with 
such a typical term as 
$\bigl(S(u_j u_k)\bigr)({\bar Z}^a\partial_{\alpha\beta}^2 u_i)
\partial_t{\bar Z}^a Su_i$ 
$(|a|=2)$. 
We use (\ref{lem314}) to deal with ${\bar Z}^a\partial_t^2 u_i$, 
and then we use (\ref{lem311}) to get 
$|u({\bar Z}^{a'}u)({\bar Z}^{a''}u)|\leq C|u|$ 
for $|a'|,|a''|\leq 1$ 
and $|u^2 {\bar Z}^a u|\leq C|u|$ for $|a|\leq 2$. 
We thus obtain by (\ref{hypart1}) with $\theta=0$ and (\ref{3ineq1})
\begin{align}\label{ji3july2}
\|&
\chi_1
\bigl(S(u_j(t) u_k(t))\bigr)
({\bar Z}^a\partial_{\alpha\beta}^2 u_i(t))
\partial_t{\bar Z}^a Su_i(t)
\|_{L^1({\mathbb R}^3)}\\
&
\leq
C
\langle t\rangle^{-1}
\|
r^{1/2}
\langle t-r\rangle
u(t)
\|_{L^\infty({\mathbb R}^3)}
\|Su(t)\|_{L^\infty({\mathbb R}^3)}
L(u(t))^2
\nonumber\\
&
+
C
\langle t\rangle^{-2}
\|
\langle t-r\rangle u(t)
\|_{L^6({\mathbb R}^3)}^2
\|Su(t)\|_{L^6({\mathbb R}^3)}
N_4(u(t))
\nonumber\\
&
\leq
C
\langle t\rangle^{-1}
\bigl(
M_2(u(t))
+
N_2(u(t))
\bigr)
M_3(u(t))
L(u(t))^2\nonumber\\
&
+
C
\langle t\rangle^{-2}
\bigl(
M_1(u(t))
+
N_1(u(t))
\bigr)^2
N_2(u(t))N_4(u(t)).\nonumber
\end{align}
Here, we have used the standard Sobolev inequalities to handle 
$\|Su_j(t)\|_{L^p({\mathbb R}^3)}$, $p=\infty$, $6$. 
All the other terms can be handled in a similar way. 

\noindent{\bf$\cdot$Estimate of $\chi_1 J_{i,4}$.} We recall that 
$H_i(u,v)$ is a homogeneous polynomial of degree $3$ 
in $u$ and $v$, 
and therefore $H(u,\partial u)$ has the form of 
sum of constant multiples of 
$u_ju_ku_l$, 
$u_ju_k\partial_\gamma u_l$, 
$u_j(\partial_\beta u_k)(\partial_\gamma u_l)$, 
and 
$(\partial_\alpha u_j)(\partial_\beta u_k)(\partial_\gamma u_l)$. 
It is possible to obtain 
\begin{equation}\label{1j14312}
\|\chi_1 J_{i,4}\|_{L^1({\mathbb R}^3)}
\leq
C\langle t\rangle^{-2}
\bigl(
M_3(u(t))^3
+
N_4(u(t))^3
\bigr)
N_4(u(t))
\end{equation}
by using (\ref{3ineq1}) together with the H\"older-type inequality 
\begin{equation}\label{holdertype1}
\|\chi_1 v_1\cdots v_4\|_{L^1}
\leq
C
\langle t\rangle^{-2}
\|\langle t-r\rangle v_1\|_{L^6}
\|\langle t-r\rangle v_2\|_{L^6}
\|v_3\|_{L^6}
\|v_4\|_{L^2}
\end{equation}
(see (\ref{ji3july2})) or 
(\ref{3ineq3}), (\ref{3ineq4}) together with the H\"older-type inequality 
\begin{equation}\label{holdertype2}
\|\chi_1 v_1\cdots v_4\|_{L^1}
\leq
C
\langle t\rangle^{-2}
\|\langle t-r\rangle v_1\|_{L^\infty}
\|\langle t-r\rangle v_2\|_{L^\infty}
\|v_3\|_{L^2}
\|v_4\|_{L^2}.
\end{equation}

\noindent{\bf$\cdot$Estimate of $\chi_1 J_{i,5}$.} 
Recall that ${\bar Z}^a$ does not contain the operator $S$. 
Therefore, using (\ref{lem312}), (\ref{lem313}), and (\ref{lem3145}) 
to deal with $\partial_t^2{\bar Z}^a u$ 
$(|a|\leq 2)$ 
and proceeding as we have done in dealing with $\chi_1 J_{i,4}$ 
just above, we easily obtain
\begin{equation}\label{ji52019july261736}
\|\chi_1 J_{i,5}\|_{L^1({\mathbb R}^3)}
\leq
C
\langle t\rangle^{-2}
\bigl(
M_3(u(t))^3+N_4(u(t))^3
\bigr)
N_4(u(t)).
\end{equation}
\subsection{Estimate over the set $\{x\in{\mathbb R}^3:|x|>(1+t)/2\}$} 
In contrast with the former subsection, 
we fully exploit the null condition. 
We start with the estimate of 
$\chi_2 q_1$. 
As for the third term on the right-hand side of (\ref{dec3}), 
we basically follow the argument in \cite{HZ2019}. 
Namely, 
we first employ (\ref{null25}) and then 
(\ref{tv1}), (\ref{lem311})--(\ref{lem312}), (\ref{sob25}), and 
(\ref{hypart1})--(\ref{hyderivative}) with $\theta=(1/2)-\eta$ to get
\begin{align}\label{feb71}
\|&\chi_2
F_i^{j,\alpha\beta\gamma}
(\partial_{t\gamma}^2 u_j(t))
(\partial_\beta{\bar Z}^a S u_i(t))
(\partial_\alpha{\bar Z}^a S u_i(t))
\|_{L^1({\mathbb R}^3)}\\
&
\leq
C\sum_{k,j}
\|
\chi_2
(T_k\partial_t u_j(t))
(\partial{\bar Z}^a S u_i (t))^2
\|_{L^1({\mathbb R}^3)}\nonumber\\
&
\hspace{0.2cm}
+C\sum_{k,j}
\|
\chi_2 (\partial_t\partial u_j(t))
(T_k{\bar Z}^a S u_i (t))
(\partial{\bar Z}^a S u_i(t))
\|_{L^1({\mathbb R}^3)}\nonumber\\
&
\leq
C\langle t\rangle^{-2}
\sum_{{|b|+|c|}\atop{+d\leq 1}}
\|r\Omega^b L^c S^d \partial_t u_j(t)\|_{L^\infty({\mathbb R}^3)}
\|\partial{\bar Z}^a S u_i(t)\|_{L^2({\mathbb R}^3)}^2\nonumber\\
&
\hspace{0.2cm}
+
C\langle t\rangle^{-1+\eta}
\|r^{1-\eta}\langle t-r\rangle^{(1/2)+\eta}
(|\partial\partial_x u(t)|+|\partial u(t)|+|u(t)|)
\|_{L^\infty({\mathbb R}^3)}\nonumber\\
&
\hspace{1cm}
\times
\|\langle t-r\rangle^{-(1/2)-\eta}
T{\bar Z}^a Su_i(t)\|_{L^2({\mathbb R}^3)}
\|\partial{\bar Z}^a S u_i(t)\|_{L^2({\mathbb R}^3)}\nonumber\\
&
\leq
C\langle t\rangle^{-2}N_4(u(t))^3
+
C
\langle t\rangle^{-1+\eta}
\bigl(
M_2(u(t))
+
N_4(u(t))
\bigr)
G(u(t))N_4(u(t)).\nonumber
\end{align}
Using (\ref{null24}) in place of (\ref{null25}) 
and repeating the same argument as above, 
we have a similar bound for 
the first term on the right-hand side of (\ref{dec3}). 

As for the second and 
the fourth terms on the right-hand side of (\ref{dec3}), 
we use the elementary bound
\begin{equation}\label{feb72}
|\partial G(u,\partial u)|
\leq
C(|u|+|\partial u|)(|\partial u|+|\partial^2 u|)
\end{equation}
and employ (\ref{lem311})--(\ref{lem312}) to handle $\partial_t^2 u$. 
We get by (\ref{sob25})
\begin{align}\label{feb73}
\|&\chi_2 (\partial G(u,\partial u))
|\partial{\bar Z}^a Su_i(t)|^2\|_{L^1({\mathbb R}^3)}\\
&
\leq
\|
\chi_2(|u(t)|+|\partial u(t)|)
(|u(t)|+|\partial u(t)|+|\partial\partial_x u(t)|)
\|_{L^\infty({\mathbb R}^3)}
N_4(u(t))^2\nonumber\\
&
\leq
C
\langle t\rangle^{-2}
\bigl(
M_2(u(t))+N_4(u(t))
\bigr)^2
N_4(u(t))^2.\nonumber
\end{align}
Suitably modifying the argument in 
(\ref{feb71})--(\ref{feb73}), 
we also obtain
\begin{align}
\|&
\chi_2 g'(t-r) q_2
\|_{L^1({\mathbb R}^3)}\\
&
\leq
C\langle t\rangle^{-3/2}N_3(u(t))N_4(u(t))^2
+
C\langle t\rangle^{-1}
N_3(u(t))G(u(t))N_4(u(t))\nonumber\\
&
\hspace{0.2cm}
+
C
\langle t\rangle^{-2}
\bigl(
M_2(u(t))+N_3(u(t))
\bigr)^2
N_4(u(t))^2.\nonumber
\end{align}
The estimate of $\chi_2 q$ has been finished. 

\noindent{\bf$\cdot$Estimate of $\chi_2 J_{i,1}$.} We basically follow 
the corresponding argument in \cite{HZ2019}. 
Using (\ref{null23}), 
we get
\begin{align}\label{2019aug31459}
\|&\chi_2 J_{i,1}\|_{L^1({\mathbb R}^3)}\\
&
\leq
\sum\!{}^{'}
\|
\chi_2{\hat F}_i^{j,\alpha\beta\gamma}
(\partial_\gamma{\bar Z}^{a'}S^{d'}u_j(t))
(\partial_{\alpha\beta}^2{\bar Z}^{a''}S^{d''}u_i(t))
(\partial_t{\bar Z}^aSu_i(t))
\|_{L^1({\mathbb R}^3)}\nonumber\\
&
\leq
C\sum_{k,j}
\biggl(
\|
\chi_2
(T_k{\bar Z}^{a'}S^{d'}u_j(t))
(\partial^2{\bar Z}^{a''}S^{d''}u_i(t))
\|_{L^2({\mathbb R}^3)}\nonumber\\
&
\hspace{2cm}
+
\|
\chi_2
(\partial{\bar Z}^{a'}S^{d'}u_j(t))
(T_k\partial{\bar Z}^{a''}S^{d''}u_i(t))
\|_{L^2({\mathbb R}^3)}
\biggr)
N_4(u(t))\nonumber\\
&
=:
C\sum_{k,j}(K_1+K_2)N_4(u(t)).\nonumber
\end{align}
Before proceeding, 
we recall that 
$|a'|+|a''|\leq 2$, $d'+d''\leq 1$, 
and $|a''|+d''\leq 2$. 
It suffices to discuss only the case $d'+d''=1$; 
the argument becomes easier otherwise. 

\noindent\underline{Case 1. $d'=1$, $d''=0$.} 

\noindent\underline{Case 1-1. $|a'|=0$, $|a''|\leq 2$.} 
Obviously, it suffices to handle only the case $|a''|=2$. 
Using (\ref{lem311}), (\ref{lem314}), (\ref{trd}), and 
(\ref{hypart1})--(\ref{hyderivative}) 
with $\theta=(1/2)-\eta$, we obtain 
\begin{align}
K_1
&
\leq
C\langle t\rangle^{-1}
\|
r\langle t-r\rangle^{-1}T_k S u_j(t)
\|_{L^\infty({\mathbb R}^3)}
\|
\langle t-r\rangle\partial_x\partial{\bar Z}^{a''}u(t)
\|_{L^2({\mathbb R}^3)}\\
&
+
C\langle t\rangle^{-1+\eta}
\|
\langle t-r\rangle^{-(1/2)-\eta}
T_kS u_j(t)
\|_{L^2_r L^4_\omega}\nonumber\\
&
\hspace{2cm}
\times
\sum_{|b|\leq 2}
\|
r^{1-\eta}
\langle t-r\rangle^{(1/2)+\eta}
\partial{\bar Z}^b u(t)
\|_{L^\infty_r L^4_\omega}
\nonumber\\
&
+
C\langle t\rangle^{-1+\eta}
\|
\langle t-r\rangle^{-(1/2)-\eta}
T_kS u_j(t)
\|_{L^2({\mathbb R}^3)}
\|
r^{1-\eta}
\langle t-r\rangle^{(1/2)+\eta} u(t)
\|_{L^\infty({\mathbb R}^3)}\nonumber\\
&
\leq
C\langle t\rangle^{-1}G(u(t))N_4(u(t))
+
C\langle t\rangle^{-1+\eta}G(u(t))
\bigl(
N_4(u(t))
+
M_2(u(t))
\bigr).\nonumber
\end{align}
Here, to handle 
$\|
r\langle t-r\rangle^{-1}T_k S u_j(t)
\|_{L^\infty({\mathbb R}^3)}$, 
we have used (see, e.g., (27), (28) in \cite{Zha})
\begin{equation}
[\Omega_{ij},T_k]
=
\delta_{kj}T_i
-
\delta_{ki}T_j,
\quad
\partial_r T_i
=
\sum_{k=1}^3\frac{x_k}{r}T_i\partial_k
\end{equation}
together with (\ref{sob25}). 
It is easy to get by (\ref{sob25}) and (\ref{tv1})
\begin{align}
K_2&\leq
C\langle t\rangle^{-2}
\|r\partial Su_j(t)\|_{L^\infty({\mathbb R}^3)}
\|
r
T_k\partial{\bar Z}^{a''}u_i(t)\|_{L^2({\mathbb R}^3)}\\
&
\leq
C\langle t\rangle^{-2}N_4(u(t))^2.\nonumber
\end{align}
\noindent\underline{Case 1-2. $|a'|\leq 1$ and $|a''|\leq 1$.} 
Employing 
$\|r\langle t-r\rangle^{-1}T_k{\bar Z}^{a'}
Su_j(t)\|_{L^\infty_r L^4_\omega}$ 
and 
$\|\langle t-r\rangle\partial_x\partial{\bar Z}^{a''}u_i
(t)\|_{L^2_r L^4_\omega}$ 
and naturally modifying the argument in Case 1-1, 
we get the same bound for $K_1$ as in Case 1-1. 
Also, employing 
$\|r\partial{\bar Z}^{a'}Su_j(t)\|_{L^\infty_r L^4_\omega}$ 
and 
$\|
r
T_k\partial{\bar Z}^{a''}u_i(t)\|_{L^2_r L^4_\omega}$, 
we get the same bound for $K_2$ as in Case 1-1. 

\noindent\underline{Case 1-3. $|a'|\leq 2$ and $|a''|=0$.} 
Using (\ref{lem311}), (\ref{lem312}) and 
(\ref{hypart1})--(\ref{hyderivative}) with $\theta=(1/2)-\eta$, 
we easily obtain
\begin{equation}
K_1
\leq
C\langle t\rangle^{-1+\eta}
G(u(t))N_4(u(t))
+
C\langle t\rangle^{-1+\eta}
G(u(t))M_2(u(t)).
\end{equation}
Also, using 
(\ref{tv1}) first and then (\ref{sob23}), we easily get 
\begin{equation}
K_2
\leq
C
\langle t\rangle^{-2}
N_4(u(t))
\|r\bigl(rT_k\partial u_i(t)\bigr)\|_{L^\infty({\mathbb R}^3)}
\leq
C\langle t\rangle^{-2}N_4(u(t))^2.
\end{equation}

\noindent\underline{Case 2. $d'=0$ and $d''=1$.} 

\noindent\underline{Case 2-1. $|a'|\leq 1$ and $|a''|\leq 1$.} 
We employ (\ref{lem313}), (\ref{lem314}) together with 
(\ref{product320}) to get 
\begin{align}
K_1&\leq
\|
\chi_2
(T_k {\bar Z}^{a'} u_j(t))
(\partial^2 {\bar Z}^{a''}S u_i(t))
\|_{L^2({\mathbb R}^3)}\\
&
\leq
C
\langle t\rangle^{-3/2}
\|
r^{1/2}
\bigl(
rT_k{\bar Z}^{a'} u_j(t)
\bigr)
\|_{L^\infty({\mathbb R}^3)}
\bigl(
N_4(u(t))
+
M_1(u(t))
\bigr)
\nonumber\\
&
\leq
C
\langle t\rangle^{-(3/2)}
N_4(u(t))
\bigl(
M_1(u(t))
+
N_4(u(t))
\bigr).\nonumber
\end{align}
Here, we have used (\ref{tv1}), (\ref{sob23}), (\ref{conf2}). 
As for $K_2$, we employ (\ref{hyderivative}) with $\theta=(1/2)-\eta$ and 
easily get 
\begin{equation}\label{troublesome}
K_2
\leq
C\langle t\rangle^{-1+\eta}
N_4(u(t))
G(u(t)).
\end{equation}
It should be noted that 
this is the one of the places where 
we encounter the norm 
$\|
\langle t-r\rangle^{-(1/2)-\eta}
T_j\partial_t{\bar Z}^aL^cS^d u_i(t)
\|_{L^2({\mathbb R}^3)}$ 
$(|a|+|c|+d\leq 2, |c|+d\leq 1)$. 

\noindent\underline{Case 2-2. $|a'|\leq 2$ and $|a''|=0$.} 
Using the $L_r^\infty L_\omega^4$-norm 
and the $L_r^2 L_\omega^4$-norm, 
we naturally modify the argument in the above case 
to get 
the same bound for $K_1$ and $K_2$ 
as in Case 2-1. We have finished the estimate of $\chi_2 J_{i,1}$.

\noindent{\bf$\cdot$Estimate of $\chi_2 J_{i,2}$.} We need to bound 
$\|
\chi_2{\hat F}_i^{jk,\alpha\beta}
(\partial_\alpha{\bar Z}^{a'}S^{d'}u_j)
(\partial_\beta{\bar Z}^{a''}S^{d''}u_k)
\|_{L^2({\mathbb R}^3)}$ 
for 
$|a'|+|a''|\leq 2$, $d'+d''\leq 1$. 
Obviously, we may focus on the case 
$d'+d''=1$. 
It follows from (\ref{null26}) that 
\begin{align}
\|&
\chi_2{\hat F}_i^{jk,\alpha\beta}
(\partial_\alpha{\bar Z}^{a'}S^{d'}u_j)
(\partial_\beta{\bar Z}^{a''}S^{d''}u_k)
\|_{L^2({\mathbb R}^3)}\\
&
\leq
C\sum_{j,k,l}
\bigl(
\|
\chi_2
(T_l{\bar Z}^{a'}S^{d'}u_j)
(\partial{\bar Z}^{a''}S^{d''}u_k)
\|_{L^2({\mathbb R}^3)}\nonumber\\
&
\hspace{2cm}
+
\|
\chi_2
(\partial{\bar Z}^{a'}S^{d'}u_j)
(T_l{\bar Z}^{a''}S^{d''}u_k)
\|_{L^2({\mathbb R}^3)}
\bigr).\nonumber
\end{align}
Due to symmetry, 
we may suppose $d'=1$ and $d''=0$. 
When $|a'|=0$ and $|a''|\leq 2$ or 
$|a'|\leq 1$ and $|a''|\leq 1$, 
we employ 
the $L_r^2 L_\omega^4$-norm 
and the $L_r^\infty L_\omega^4$-norm. 
We get by (\ref{hyderivative}) with $\theta=(1/2)-\eta$
\begin{equation}
\|
\chi_2
(T_l{\bar Z}^{a'}Su_j(t))
(\partial{\bar Z}^{a''}u_k(t))
\|_{L^2({\mathbb R}^3)}
\leq
C\langle t\rangle^{-1+\eta}
G(u(t))N_4(u(t)).
\end{equation}
Also, we get by (\ref{sob23}), (\ref{tv1})
\begin{equation}
\|
\chi_2
(\partial{\bar Z}^{a'}Su_j(t))
(T_l{\bar Z}^{a''}u_k(t))
\|_{L^2({\mathbb R}^3)}
\leq
C\langle t\rangle^{-3/2}
N_4(u(t))^2.
\end{equation}
When $|a'|\leq 2$ and $|a''|=0$, 
we have only to modify the argument just above and 
employ 
the $L^2({\mathbb R}^3)$-norm and the $L^\infty({\mathbb R}^3)$-norm. 
We have finished the estimate of $\chi_2 J_{i,2}$.

\noindent{\bf$\cdot$Estimate of 
$\chi_2 J_{i,3}$, $\chi_2 J_{i,4}$, and  $\chi_2 J_{i,5}$.} 
Using (\ref{sob24}) and (\ref{sob25}), we obtain 
\begin{equation}\label{1818}
\sum_{k=3}^5
\|
\chi_2 J_{i,k}
\|_{L^1({\mathbb R}^3)}
\leq
C\sum_{m=0}^4
\langle t\rangle^{-2}
\bigl(
N_4(u(t))^{2-(m/2)}M_3(u(t))^{m/2}
\bigr)
N_4(u(t))^2.
\end{equation}
The proof is direct and is therefore omitted. 

Now we are in a position to complete the proof of 
Proposition \ref{ghostenergyestimate}. 
We first note that 
the function $g=g(\rho)$ 
$(\rho\in{\mathbb R})$ 
is bounded (see (\ref{howg})), 
and hence there exists a positive constant $c$ such that 
$c\leq e^g\leq c^{-1}$ 
$(\rho\in{\mathbb R})$. 
We also note that 
$g'$ is a negative function, 
and it therefore follows from 
(\ref{dec2}), (\ref{1817})--(\ref{1818}) that 
for $i=1,\dots,N$ and $|a|=2$, 
$$
E_1({\bar Z}^a S u_i(t))
+
\sum_{j=1}^3
\int_0^t
\|
\langle \tau-r\rangle^{-(1/2)-\eta}
T_j{\bar Z}^a S u_i(\tau)
\|_{L^2({\mathbb R}^3)}^2
d\tau
$$
is estimated from above 
by the right-hand side of (\ref{genergyest}). 
(Strictly speaking, the term $CD(f,g)^6$ there plays no role at this moment.) 
We should mention how to 
estimate 
$
\|
\langle\tau-r\rangle^{-(1/2)-\eta}
T_j\partial_t{\bar Z}^a S u_i
\|_{L^2((0,t)\times{\mathbb R}^3)}
$ 
for $|a|=1$. 
In (\ref{dec1})--(\ref{dec8}), 
we replace 
${\bar Z}^aS$ $(|a|=2)$ 
with $\partial_t{\bar Z}^aS$ 
$(|a|=1)$, 
accordingly modifying 
${\bar Z}^{a'}S^{d'}$, 
${\bar Z}^{a''}S^{d''}$ suitably. 
Also, at the last term on the left-hand side of (\ref{dec1}) 
and in (\ref{july6}), 
we replace ${\bar Z}^a$ $(|a|=2)$ with 
$\partial_t {\bar Z}^a$ 
$(|a|=1)$. 
Though we then 
encounter such a little troublesome terms 
as 
$\partial_t^3{\bar Z}^a S^d u_i$ 
$(|a|+d\leq 1)$, 
we can rely upon (\ref{lem3125}) and (\ref{lem3145}) 
to handle such terms 
in the same way as we have done above. 
Also, note that the term $CD(f,g)^3$ naturally comes up 
from $\|\partial_t^2{\bar Z}^a S u_i(0)\|_{L^2({\mathbb R}^3)}$ 
$(|a|=1)$, 
and this is the reason why we need $CD(f,g)^6$ on the right-hand side of 
(\ref{genergyest}). 
Finally, we also mention that another troublesome term 
$T_j\partial_t^2{\bar Z}^a S^d u_i$ 
$(|a|+d\leq 1)$ comes up 
when we rely upon Lemma \ref{lemmanullt}. 
We can get over this difficulty by employing 
(\ref{lem315}) and (\ref{lem316}). 
We have finished the estimate of 
\begin{align*}
\sum_{|a|=2}
E_1({\bar Z}^a S u_i(t))
&+
\sum_{j=1}^3
\sum_{|a|=2}
\int_0^t
\|
\langle \tau-r\rangle^{-(1/2)-\eta}
T_j{\bar Z}^a S u_i(\tau)
\|_{L^2({\mathbb R}^3)}^2
d\tau\\
&
+
\sum_{j=1}^3
\sum_{|a|=1}
\int_0^t
\|
\langle \tau-r\rangle^{-(1/2)-\eta}
T_j\partial_t{\bar Z}^a S u_i(\tau)
\|_{L^2({\mathbb R}^3)}^2
d\tau.
\end{align*}
The other terms appearing on the left-hand side of (\ref{genergyest}) 
can be estimated in a similar way, 
and we have therefore completed the proof of Proposition \ref{ghostenergyestimate}. 
\end{proof}
\section{Bound for $M_3(u(t))$}\label{sectionconformal}
The main purpose of this section is to prove\,:
\begin{proposition}\label{propositionm3}
Suppose that initial data $(\ref{data1})$ is smooth and compactly supported, 
and suppose that the local solution $u$ satisfies $(\ref{lem311})$ in 
some interval $(0,T)$. 
Then the following inequality holds for all $t\in (0,T):$ 
\begin{align}\label{m3conf313}
M_3(u(t))
\leq&
C
\sum_{|a|\leq 2}
\bigl(
\|
{\bar Z}^a f
\|_{L^2({\mathbb R}^3)}
+
\|
|x|
\partial_x {\bar Z}^a f
\|_{L^2({\mathbb R}^3)}
+
\|
|x|
{\bar Z}^a g
\|_{L^2({\mathbb R}^3)}
\bigr)\\
&
+
C\int_0^t
\langle\tau\rangle^{-1}
\bigl(
M_3(u(\tau)
+
N_3(u(\tau))
\bigr)
N_4(u(\tau))
d\tau
\nonumber\\
&
+
C\int_0^t
\langle\tau\rangle^{-3/2}
M_3(u(\tau))
N_2(u(\tau))^3
d\tau\nonumber\\
&
+
C\int_0^t
\langle\tau\rangle^{-1}
\bigl(
N_3(u(\tau))^2
+
X_2(u(\tau))^2
\bigr)
N_4(u(\tau))d\tau\nonumber\\
&
+
C\int_0^t
\langle\tau\rangle^{-2}
\bigl(
M_3(u(\tau))^2
+
N_3(u(\tau))^2
\bigr)
N_4(u(\tau))
d\tau\nonumber\\
&
+
C\int_0^t
\langle\tau\rangle^{-2}
M_3(u(\tau))^3
d\tau
+
C\int_0^t
\langle\tau\rangle^{-1}
X_2(u(\tau))^3d\tau.\nonumber
\end{align}
\end{proposition}
In view of the conformal energy estimate (\ref{confest}), 
it amounts to bounding
\begin{align}
&
\|
(t+|x|)
{\bar Z}^a F_i^{j,\alpha\beta\gamma}
(\partial_\gamma u_j)
(\partial_{\alpha\beta}^2 u_i)
\|_{L^2({\mathbb R}^3)},\label{conf41}\\
&
\|
(t+|x|)
{\bar Z}^a F_i^{jk,\alpha\beta}
(\partial_\alpha u_j)
(\partial_\beta u_k)
\|_{L^2({\mathbb R}^3)},\label{conf42}\\
&
\|
(t+|x|)
{\bar Z}^a
\bigl(
G_i^{\alpha\beta}(u,\partial u)
\partial_{\alpha\beta}^2 u_i
\bigr)
\|_{L^2({\mathbb R}^3)},\label{conf43}\\
&
\|
(t+|x|)
{\bar Z}^a
H_i(u,\partial u) 
\|_{L^2({\mathbb R}^3)}\label{conf44}
\end{align}
(see (\ref{eq1}), (\ref{nlt1}), (\ref{nlt2})) for $|a|\leq 2$. 
As in the previous section, 
we treat them by 
considering the $L^2$ norm 
over the set $\{x\in{\mathbb R}^3\,:\,|x|<(1+t)/2\}$ 
and its complement set, separately. 
Furthermore, 
when considering the $L^2$ norm over the former set, 
we deal with the case $t<3$ and $t>3$, separately. 

\noindent{\bf $\cdot$ $L^2$ norm over the set 
$\{x\in{\mathbb R}^3\,:\,|x|<(1+t)/2\}$ with $t<3$.} 
Obviously, it suffices to discuss how to bound 
the $L^2({\mathbb R}^3)$ norm of 
$\partial_x^a F_i^{j,\alpha\beta\gamma}
(\partial_\gamma u_j)
(\partial_{\alpha\beta}^2 u_i)$, 
$\partial_x^a F_i^{jk,\alpha\beta}
(\partial_\alpha u_j)
(\partial_\beta u_k)$, 
$\partial_x^a
\bigl(
G_i^{\alpha\beta}(u,\partial u)
\partial_{\alpha\beta}^2 u_i
\bigr)$, 
$\partial_x^a
H_i(u,\partial u)$ 
for $|a|\leq 2$. 
Using (\ref{lem311}), (\ref{lem312}), (\ref{lem313}), and 
(\ref{lem314}), we easily get 
\begin{align}
&\sum_{|a|\leq 2}
\biggl(
\|
\partial_x^a F_i^{j,\alpha\beta\gamma}
(\partial_\gamma u_j)
(\partial_{\alpha\beta}^2 u_i)
\|_{L^2({\mathbb R}^3)}
+
\|
\partial_x^a F_i^{jk,\alpha\beta}
(\partial_\alpha u_j)
(\partial_\beta u_k)
\|_{L^2({\mathbb R}^3)}
\biggr)\\
&
\hspace{0.2cm}
\leq
C
\bigl(
M_1(u(t))
+
N_3(u(t))
\bigr)
N_4(u(t)).\nonumber
\end{align}
Moreover, 
using not only (\ref{lem311}), (\ref{lem312}), (\ref{lem313}), and 
(\ref{lem314}) 
but also 
the H\"older inequality 
$\|uvw\|_{L^2({\mathbb R}^3)}
\leq
\|u\|_{L^6({\mathbb R}^3)}
\|v\|_{L^6({\mathbb R}^3)}
\|w\|_{L^6({\mathbb R}^3)}$ 
together with the Sobolev-type inequality 
$\|v\|_{L^6({\mathbb R}^3)}
\leq
C\|\nabla v\|_{L^2({\mathbb R}^3)}$, 
we easily get
\begin{align}
&
\sum_{|a|\leq 2}
\biggl(
\|
\partial_x^a
\bigl(
G_i^{\alpha\beta}(u,\partial u)
\partial_{\alpha\beta}^2 u_i
\bigr)
\|_{L^2({\mathbb R}^3)}
+
\|
\partial_x^a
H_i(u,\partial u)
\|_{L^2({\mathbb R}^3)}
\biggr)\\
&
\hspace{0.2cm}
\leq
C
\bigl(
M_1(u(t))
+
N_3(u(t))
\bigr)^2
N_4(u(t)).\nonumber
\end{align}
Next let us bound 
the $L^2$ norm of (\ref{conf41})--(\ref{conf44}) 
over the set $\{x\in{\mathbb R}^3\,:\,|x|<(1+t)/2\}$ with $t>3$ 
and the one $\{x\in{\mathbb R}^3\,:\,|x|>(1+t)/2\}$ with $t>0$. 

\noindent{\bf $\cdot$ Estimate of (\ref{conf41}).} 
On account of (\ref{null21}), 
we need to estimate 
$$
(t+|x|){\hat F}_i^{j,\alpha\beta\gamma}
(\partial_\gamma{\bar Z}^{a'}u_j)
(\partial_{\alpha\beta}^2{\bar Z}^{a''}u_i)
$$ 
$(|a'|+|a''|\leq 2)$, 
where the new coefficients ${\hat F}_i^{j,\alpha\beta\gamma}$, 
which in fact may depend on $a'$ and $a''$, 
satisfy the null condition. 
We first note that due to the null condition, 
(\ref{null23}), and (\ref{tv1})--(\ref{tv2})
we have
\begin{align}\label{march45}
(t&+|x|)|{\hat F}_i^{j,\alpha\beta\gamma}
(\partial_\gamma{\bar Z}^{a'}u_j)
(\partial_{\alpha\beta}^2{\bar Z}^{a''}u_i)|\\
&
\leq
C
\biggl(
\sum_{{|b|+|c|}\atop{+d=1}}
|\Omega^b L^c S^d{\bar Z}^{a'}u|
\biggr)
|\partial^2{\bar Z}^{a''}u_i|\nonumber\\
&
\hspace{0.2cm}
+
C|\partial{\bar Z}^{a'}u|
\biggl(
\sum_{{|b|+|c|}\atop{+d=1}}
|\Omega^b L^c S^d\partial{\bar Z}^{a''}u_i|
\biggr)
=:P_1+P_2.\nonumber
\end{align}
We carry out the estimate of $P_1$, $P_2$ 
over the set $\{x\in{\mathbb R}^3\,:\,|x|<(1+t)/2\}$ 
with $t>3$ and 
the one $\{x\in{\mathbb R}^3\,:\,|x|>(1+t)/2\}$ with $t>0$, 
separately. 

When $|a''|=0$ (and hence $|a'|\leq 2$), 
we get by (\ref{ddd}) and the Sobolev embedding 
\begin{align}\label{march46}
\|\chi_1 P_1\|_{L^2({\mathbb R}^3)}
&
\leq
\langle t\rangle^{-1}M_3(u(t))
\|
|t-r|\partial (\partial u(t))
\|_{L^\infty({\mathbb R}^3)}\\
&
\leq
C\langle t\rangle^{-1}
M_3(u(t))
\sum_{|b|+|c|+d=1}
\|
\Omega^b L^c S^d(\partial u(t))
\|_{L^\infty({\mathbb R}^3)}\nonumber\\
&
\leq
C\langle t\rangle^{-1}
M_3(u(t))N_4(u(t)).\nonumber
\end{align}
Here we have used the inequality 
\begin{equation}\label{2019shingengo111}
t-r\geq t-\frac{1+t}{2}
=
\frac{t}{4}
+
\frac{t}{4}
-\frac12
\geq
\frac{t}{4}
+
\frac14,
\end{equation}
which holds on the set 
$\{x\in{\mathbb R}^3\,:\,|x|<(1+t)/2\}$ 
with $t>3$. 
When $|a''|=1$ (and hence $|a'|\leq 1$) 
or $|a''|=2$ (and hence $|a'|=0$), 
we get by (\ref{2019shingengo111}), 
(\ref{ddd}), and the Sobolev embedding 
\begin{align}\label{march47}
\|&\chi_1 P_1\|_{L^2({\mathbb R}^3)}\\
&
\leq
C
\langle t\rangle^{-1}
\biggl(
\sum_{{|b|+|c|}\atop{+d=1}}
\|
\Omega^b L^c S^d {\bar Z}^{a'} u(t)
\|_{L^\infty({\mathbb R}^3)}
\biggr)
\|
|t-r|\partial(\partial{\bar Z}^{a''}u(t))
\|_{L^2({\mathbb R}^3)}\nonumber\\
&
\leq
C\langle t\rangle^{-1}
\bigl(
M_2(u(t))
+
N_3(u(t))
\bigr)N_4(u(t)).\nonumber
\end{align}
As for $P_2$, 
we use (\ref{3ineq2}) to get 
for $|a''|=0$ (and hence $|a'|\leq 2$)
\begin{align}\label{reiwa1234}
\|&
\chi_1 P_2
\|_{L^2({\mathbb R}^3)}\\
&
\leq
C\langle t\rangle^{-1}
\|
\langle t-r\rangle
\partial{\bar Z}^{a'} u(t)
\|_{L^6({\mathbb R}^3)}
\biggl(
\sum_{{|b|+|c|}\atop{+d=1}}
\|
\Omega^b L^c S^d \partial u_i (t)
\|_{L^3({\mathbb R}^3)}
\biggr)\nonumber\\
&
\leq
C\langle t\rangle^{-1}
N_3(u(t))N_4(u(t)).\nonumber
\end{align}
For $|a''|=1$ 
(and hence $|a'|\leq 1$) or 
$|a''|=2$ (and hence $|a'|=0$), 
we apply (\ref{3ineq4}) to 
$\|
\langle t-r\rangle
\partial {\bar Z}^{a'} u(t)
\|_{L^\infty({\mathbb R}^3)}$ 
and obtain the same bound for $\|\chi_1 P_2\|_{L^2({\mathbb R}^3)}$ 
as in (\ref{reiwa1234}). 

Turning our attention to the estimate of 
$\chi_2 P_1$ and $\chi_2 P_2$, 
we get by (\ref{lem312}) and (\ref{sob23}), 
for $|a''|=0$ (and hence $|a'|\leq 2$) 
\begin{align}\label{march410}
\|
\chi_2 P_1
\|_{L^2({\mathbb R}^3)}
&
\leq
C\langle t\rangle^{-1}
M_3(u(t))
\sum_{|a|\leq 1}
\|
r\partial\partial_x^a u(t)
\|_{L^\infty({\mathbb R}^3)}\\
&
\hspace{0.1cm}
+
C\langle t\rangle^{-3/2}
M_3(u(t))
\|
r^{1/2}u(t)
\|_{L^\infty({\mathbb R}^3)}^3
\nonumber\\
&
\leq
C\langle t\rangle^{-1}
M_3(u(t))N_4(u(t))
+
C\langle t\rangle^{-3/2}
M_3(u(t))
N_2(u(t))^3.\nonumber
\end{align}
For $|a''|=1$ 
(and hence $|a'|\leq 1$), 
we get by 
(\ref{lem313}) and (\ref{sob23})
\begin{align}\label{march111}
\|&
\chi_2 P_1
\|_{L^2({\mathbb R}^3)}\\
&
\leq
C\langle t\rangle^{-1}
\biggl(
\sum_{{|b|+|c|}\atop{+d=1}}
\|
r\Omega^b L^c S^d {\bar Z}^{a'}u(t)
\|_{L_r^\infty L_\omega^4}
\biggr)
\biggl(
\sum_{|a|\leq 2}
\|
\partial{\bar Z}^a u(t)
\|_{L^2_r L^4_\omega}
\biggr)\nonumber\\
&
\hspace{0.2cm}
+
C
\langle t\rangle^{-3/2}
\biggl(
\sum_{{|b|+|c|}\atop{+d=1}}
\|
\Omega^b L^c S^d {\bar Z}^{a'}u(t)
\|_{L_r^2 L_\omega^4}
\biggr)
\|
r^{1/2}
u(t)
\|_{L^\infty({\mathbb R}^3)}^2\nonumber\\
&
\hspace{2cm}
\times
\biggl(
\sum_{|a|\leq 1}
\|
r^{1/2}{\bar Z}^a u(t)
\|_{L^\infty_r L^4_\omega}
\biggr)\nonumber\\
&
\leq
C\langle t\rangle^{-1}
\bigl(
N_3(u(t))^{1/2}
M_3(u(t))^{1/2}
\bigr)
N_4(u(t))
+
C\langle t\rangle^{-3/2}
M_3(u(t))N_2(u(t))^3.\nonumber
\end{align}
For $|a''|=2$ (and hence $|a'|=0$), 
we employ (\ref{lem314}) and modify 
the argument above to get 
the same estimate as in (\ref{march111}). 
As for $\chi_2 P_2$, it is easy to get
\begin{equation}
\|
\chi_2 P_2
\|_{L^2({\mathbb R}^3)}
\leq
C\langle t\rangle^{-1}
N_3(u(t))N_4(u(t)).
\end{equation}
\noindent{\bf$\cdot$Estimate of (\ref{conf42}).} 
In view of (\ref{null22}), 
we need to deal with
$$
(t+|x|)
{\hat F}_i^{jk,\alpha\beta}
(\partial_\alpha{\bar Z}^{a'}u_j)
(\partial_\beta{\bar Z}^{a''}u_k),
\quad
|a'|+|a''|\leq 2,
$$
where the new coefficients 
${\hat F}_i^{jk,\alpha\beta}$, 
which may depend on $a'$ and $a''$, 
satisfy the null condition.  
We then have, as in (\ref{march45})
\begin{align}\label{march7413}
(&t+|x|)
|{\hat F}_i^{jk,\alpha\beta}
(\partial_\alpha{\bar Z}^{a'}u_j)
(\partial_\beta{\bar Z}^{a''}u_k)|\\
&
\leq
\biggl(
\sum_{{|b|+|c|}\atop{+d=1}}
|\Omega^b L^c S^d {\bar Z}^{a'}u|
\biggr)
|\partial{\bar Z}^{a''}u|
+
C|\partial{\bar Z}^{a'}u|
\biggl(
\sum_{{|b|+|c|}\atop{+d=1}}
|\Omega^b L^c S^d {\bar Z}^{a''}u|
\biggr)\nonumber\\
&
=:
P_3+P_4.\nonumber
\end{align}
By symmetry, we have only to deal with $P_3$. 
When $|a''|=0$ (and hence $|a'|\leq 2$) 
or $|a''|=1$ (and hence $|a'|\leq 1$), 
we get by (\ref{3ineq4})
\begin{align}
\|
\chi_1 P_3
\|_{L^2({\mathbb R}^3)}
&
\leq
C
\langle t\rangle^{-1}
M_3(u(t))
\|
\langle t-r\rangle\partial{\bar Z}^{a''}u(t)
\|_{L^\infty({\mathbb R}^3)}\\
&
\leq
C
\langle t\rangle^{-1}
M_3(u(t))N_4(u(t)).\nonumber
\end{align}
When $|a''|=2$ (and hence $|a'|=0$), 
we employ (\ref{3ineq2}) and obtain 
\begin{align}
\|
\chi_1 P_3
\|_{L^2({\mathbb R}^3)}
&
\leq
C\langle t\rangle^{-1}
\biggl(
\sum_{{|b|+|c|}\atop{+d=1}}
\|
\Omega^b L^c S^d u(t)
\|_{L^3({\mathbb R}^3)}
\biggr)
\|
\langle t-r \rangle
\partial{\bar Z}^{a''}u(t)
\|_{L^6({\mathbb R}^3)}\\
&
\leq
C
\langle t\rangle^{-1}
M_3(u(t))N_4(u(t)).\nonumber
\end{align}
As for $\chi_2 P_3$, 
we get in a way similar to (\ref{march410})--(\ref{march111})
\begin{align}
\|
\chi_2 P_3
\|_{L^2({\mathbb R}^3)}
\leq&
C\langle t\rangle^{-1}
M_3(u(t))
N_3(u(t))\\
&
+
C\langle t\rangle^{-1}
\bigl(
N_3(u(t))^{1/2}
M_3(u(t))^{1/2}
\bigr)
N_3(u(t)).\nonumber
\end{align}

\noindent{\bf$\cdot$ Estimate of (\ref{conf43}).} 
We can obtain for $|a|\leq 2$, $i=1,\dots,N$
\begin{align}
\|&
\chi_1
(t+|x|)
{\bar Z}^a
\bigl(
G_i^{\alpha\beta}(u,\partial u)\partial_{\alpha\beta}^2 u_i(t)
\bigr)
\|_{L^2({\mathbb R}^3)}\label{march3417}\\
&
\leq
C
\langle t\rangle^{-2}
\bigl(
M_3(u(t))
+
N_3(u(t))
\bigr)^2
N_4(u(t)),\nonumber\\
\|&
\chi_2
(t+|x|)
{\bar Z}^a
\bigl(
G_i^{\alpha\beta}(u,\partial u)\partial_{\alpha\beta}^2 u_i(t)
\bigr)
\|_{L^2({\mathbb R}^3)}\label{march3418}\\
&
\leq
C\langle t\rangle^{-1}
\bigl(
N_3(u(t))^2
+
X_2(u(t))^2
\bigr)
N_4(u(t)).\nonumber
\end{align}
For the proof of (\ref{march3417}), 
it suffices to explain 
how to handle such typical terms as 
$u_j({\bar Z}^a u_k)\partial^2 u_i$ 
and 
$u_ju_k{\bar Z}^a\partial^2 u_i$ for $|a|=2$ 
because the other terms can be treated 
in a similar way. 

Recall that we are assuming $t\geq 3$ 
when considering the estimate over the set 
$\{x\in{\mathbb R}^3\,:\,|x|<(1+t)/2\}$. 
Using (\ref{3ineq1}) and (\ref{ddd}), 
we obtain
\begin{align}
\|&
\chi_1(t+|x|)
u_j(t)
({\bar Z}^a u_k(t))
\partial^2 u_i(t)
\|_{L^2({\mathbb R}^3)}\\
&
\leq
C\langle t\rangle^{-2}
\|
\langle t-r\rangle u_j(t)
\|_{L^6({\mathbb R}^3)}
\|
\langle t-r\rangle{\bar Z}^a u_k(t)
\|_{L^6({\mathbb R}^3)}
\|
|t-r|\partial(\partial u_i(t))
\|_{L^6({\mathbb R}^3)}\nonumber\\
&
\leq
C\langle t\rangle^{-2}
\bigl(
M_3(u(t))
+
N_3(u(t))
\bigr)^2
N_3(u(t)).\nonumber
\end{align}
Also, using (\ref{3ineq3}) and (\ref{ddd}), 
we get 
\begin{align}
\|&
\chi_1
(t+|x|)
u_j(t)
u_k(t)
{\bar Z}^a\partial^2 u_i(t)
\|_{L^2({\mathbb R}^3)}\\
&
\leq
C\langle t\rangle^{-2}
\|
\langle t-r\rangle u(t)
\|_{L^\infty({\mathbb R}^3)}^2
\sum_{|b|\leq 2}
\|
|t-r|\partial
(\partial{\bar Z}^b u(t))
\|_{L^2({\mathbb R}^3)}\nonumber\\
&
\leq
C\langle t\rangle^{-2}
\bigl(
M_1(u(t))
+
N_2(u(t))
\bigr)^2
N_4(u(t)).\nonumber
\end{align}
For the proof of (\ref{march3418}), 
it suffices to explain how to 
treat $({\bar Z}^{a'}u_j)({\bar Z}^{a''}u_k)\partial^2 u_i$ 
with $|a'|=|a''|=1$ and 
$u_j({\bar Z}^a u_k)\partial^2 u_i$ 
with $|a|=2$; 
the other terms can be handled 
in a similar manner. 
Using (\ref{sob24}) and (\ref{lem311})--(\ref{lem312}), 
we get
\begin{align}
\|&
\chi_2
(t+|x|)
({\bar Z}^{a'}u_j(t))
({\bar Z}^{a''}u_k(t))
\partial^2 u_i(t)
\|_{L^2({\mathbb R}^3)}\\
&
\leq
C\langle t\rangle^{-1}
\|
r{\bar Z}^{a'}u_j(t)
\|_{L_r^\infty L_\omega^4}
\|
r{\bar Z}^{a''}u_k(t)
\|_{L_r^\infty L_\omega^4}
\|
\partial^2 u_i(t)
\|_{L_r^2 L_\omega^\infty}\nonumber\\
&
\leq
C\langle t\rangle^{-1}
\bigl(
N_2(u(t))^{1/2}
X_2(u(t))^{1/2}
\bigr)^2
\bigl(
N_4(u(t))
+
X_2(u(t))
\bigr)\nonumber\\
&
\leq
C\langle t\rangle^{-1}
\bigl(
N_2(u(t))^2
+
X_2(u(t))^2
\bigr)N_4(u(t)),\nonumber
\end{align}
where we have used the Young inequality. 
Also, using (\ref{sob25}), (\ref{lem311})--(\ref{lem312}), 
we obtain 
\begin{align}
\|&
\chi_2
(t+|x|)
u_j(t)
({\bar Z}^a u_k(t))
\partial^2 u_i(t)
\|_{L^2({\mathbb R}^3)}\\
&\leq
C\langle t\rangle^{-1}
\|
ru_j(t)
\|_{L^\infty({\mathbb R}^3)}
\|
{\bar Z}^a u_k(t)
\|_{L^2({\mathbb R}^3)}
\|
r\partial^2 u_i(t)
\|_{L^\infty({\mathbb R}^3)}\nonumber\\
&
\leq
C\langle t\rangle^{-1}
\bigl(
N_2(u(t))^{1/2}
X_2(u(t))^{1/2}
\bigr)
X_2(u(t))\nonumber\\
&
\hspace{1.5cm}
\times
\biggl(
\sum_{|b|\leq 1}
\|
r\partial\partial_x^b u(t)
\|_{L^\infty({\mathbb R}^3)}
+
\|
ru(t)
\|_{L^\infty({\mathbb R}^3)}
\biggr)\nonumber\\
&
\leq
C\langle t\rangle^{-1}
N_2(u(t))^{1/2}
X_2(u(t))^{3/2}
\bigl(
N_4(u(t))
+
N_2(u(t))^{1/2}
X_2(u(t))^{1/2}
\bigr)\nonumber\\
&
\leq
C\langle t\rangle^{-1}
\bigl(
N_2(u(t))^2
+
X_2(u(t))^2
\bigr)
N_4(u(t)),\nonumber
\end{align}
where we have used the Young inequality again. 

\noindent{\bf$\cdot$Estimate of (\ref{conf44}).} 
Repeating essentially the same argument as above, 
we can obtain for $|a|\leq 2$
\begin{align}
\|&
(t+|x|)
{\bar Z}^a 
H_i(u,\partial u)
\|_{L^2({\mathbb R}^3)}\\
&
\leq
C\langle t\rangle^{-2}
\bigl(
M_3(u(t))^3
+
N_3(u(t))^3
\bigr)
+
C\langle t\rangle^{-1}
\bigl(
N_3(u(t))^3
+
X_2(u(t))^3
\bigr).\nonumber
\end{align}
It is now obvious that (\ref{m3conf313}) 
is an immediate consequence of the estimates 
we have obtained above. 
The proof of Proposition \ref{propositionm3} 
has been finished.
\section{Bound for $X_2(u(t))$}\label{sectionx2norm}
The purpose of this section is to prove
\begin{proposition}\label{propositionx2}
Suppose that initial data $(\ref{data1})$ is smooth and compactly supported, 
and suppose that the local solution $u$ satisfies $(\ref{lem311})$ in 
some interval $(0,T)$. 
Then the following inequality holds for all $t\in (0,T):$ 
\begin{align}\label{2019estimatex2}
&X_2(u(t))\\
&
\hspace{0.1cm}
\leq
C
\biggl(
\sum_{|a|\leq 2}
\|{\bar Z}^a f\|_{L^2({\mathbb R}^3)}
+
\sum_{|a|\leq 2}
\||x|{\bar Z}^a g\|_{L^2({\mathbb R}^3)}
\biggr)\nonumber\\
&
\hspace{0.2cm}
+
C\int_0^t
\langle \tau\rangle^{-2}
\bigl(
M_3(u(\tau))+N_3(u(\tau))
\bigr)
\bigl(
M_3(u(\tau))+N_4(u(\tau))
\bigr)d\tau\nonumber\\
&
\hspace{0.2cm}
+
C\int_0^t
\langle \tau\rangle^{-2}
\bigl(
M_3(u(\tau))+N_3(u(\tau))
\bigr)^2
\bigl(
N_4(u(\tau))+X_2(u(\tau))
\bigr)d\tau\nonumber\\
&
\hspace{0.2cm}
+
C\int_0^t
\langle \tau\rangle^{-3/2}
\bigl(
M_3(u(\tau))+N_3(u(\tau))
\bigr)
\bigl(
N_4(u(\tau))+X_1(u(\tau))
\bigr)d\tau
\nonumber\\
&
\hspace{0.2cm}
+
C
\int_0^t
\langle \tau\rangle^{-3/2}
\bigl(
N_3(u(\tau))+X_2(u(\tau))
\bigr)^2
\bigl(
N_4(u(\tau))+X_2(u(\tau))
\bigr)d\tau.\nonumber
\end{align}
\end{proposition}
In view of the Li-Yu estimate (\ref{ly}) 
together with the well-known inequality 
$\||D|^{-1}v\|_{L^2({\mathbb R}^n)}
\leq
C\||x|v\|_{L^2({\mathbb R}^n)}$ 
$(n\geq 3)$, 
the proof of this proposition amounts to 
showing decay estimates of the following norms for $|a|\leq 2:$
\begin{align}
&
\|
\chi_1
{\bar Z}^a F_i^{j,\alpha\beta\gamma}
(\partial_\gamma u_j)
(\partial_{\alpha\beta}^2 u_i)
\|_{L^{6/5}({\mathbb R}^3)},
\quad
\|
\chi_2
{\bar Z}^a F_i^{j,\alpha\beta\gamma}
(\partial_\gamma u_j)
(\partial_{\alpha\beta}^2 u_i)
\|_{L_r^1 L_\omega^{4/3}({\mathbb R}^3)},\label{ly136}\\
&
\|
\chi_1
{\bar Z}^a F_i^{jk,\alpha\beta}
(\partial_\alpha u_j)
(\partial_\beta u_k)
\|_{L^{6/5}({\mathbb R}^3)},
\quad
\|
\chi_2
{\bar Z}^a F_i^{jk,\alpha\beta}
(\partial_\alpha u_j)
(\partial_\beta u_k)
\|_{L_r^1 L_\omega^{4/3}({\mathbb R}^3)},\label{ly236}\\
&
\|
\chi_1
{\bar Z}^a
\bigl(
G_i^{\alpha\beta}(u,\partial u)
\partial_{\alpha\beta}^2 u_i
\bigr)
\|_{L^{6/5}({\mathbb R}^3)},
\quad
\|
\chi_2
{\bar Z}^a
\bigl(
G_i^{\alpha\beta}(u,\partial u)
\partial_{\alpha\beta}^2 u_i
\bigr)
\|_{L_r^1 L_\omega^{4/3}({\mathbb R}^3)},\label{ly336}\\
&
\|
\chi_1
{\bar Z}^a
H_i(u,\partial u) 
\|_{L^{6/5}({\mathbb R}^3)},
\quad
\|
\chi_2
{\bar Z}^a
H_i(u,\partial u) 
\|_{L_r^1 L_\omega^{4/3}({\mathbb R}^3)}.\label{ly436}
\end{align}
\noindent{\bf$\cdot$Estimate of (\ref{ly136}).} 
We need to handle the $L^{6/5}$ norm 
for $t<3$ and $t>3$, separately. 
It is easy to get for $|a|\leq 2$
\begin{align}
\|&
\chi_1
{\bar Z}^a F_i^{j,\alpha\beta\gamma}
(\partial_\gamma u_j)
(\partial_{\alpha\beta}^2 u_i)
\|_{L^{6/5}({\mathbb R}^3)}\\
&
\leq
CN_3(u(t))
\bigl(
M_1(u(t))
+
N_4(u(t))
\bigr),\,\,0<t<3.\nonumber
\end{align}
On the other hand, for $t>3$, 
we need to handle $(t+|x|)^{-1}P_i$ 
$(i=1,2)$, as in (\ref{march45}). 
When $|a''|=0$ (and hence $|a'|\leq 2$) or 
$|a''|=1$ (and hence $|a'|\leq 1$), 
we get owing to (\ref{2019shingengo111}) and (\ref{ddd}) 
\begin{align}
\|&
\chi_1
(t+|x|)^{-1}P_1
\|_{L^{6/5}({\mathbb R}^3)}\\
&
\leq
C
\langle t\rangle^{-2}
\biggl(
\sum_{{|b|+|c|}\atop{+d=1}}
\|
\Omega^b L^c S^d{\bar Z}^{a'} u(t)
\|_{L^2({\mathbb R}^3)}
\biggr)
\|
|t-r|\partial
(\partial{\bar Z}^{a''}u(t))
\|_{L^3({\mathbb R}^3)}\nonumber\\
&
\leq
C
\langle t\rangle^{-2}
M_3(u(t))
N_4(u(t)),\,\,3<t<T.
\nonumber
\end{align}
When $|a''|=2$ (and hence $|a'|=0$), 
we get in a similar way
\begin{align}
\|&
\chi_1
(t+|x|)^{-1}P_1
\|_{L^{6/5}({\mathbb R}^3)}\\
&
\leq
C
\langle t\rangle^{-2}
\biggl(
\sum_{{|b|+|c|}\atop{+d=1}}
\|
\Omega^b L^c S^d u(t)
\|_{L^3({\mathbb R}^3)}
\biggr)
\|
|t-r|\partial
(\partial{\bar Z}^{a''}u(t))
\|_{L^2({\mathbb R}^3)}\nonumber\\
&
\leq
C
\langle t\rangle^{-2}
M_3(u(t))
N_4(u(t)),\,\,3<t<T.\nonumber
\end{align}
As for $P_2$, 
we get for $|a''|=0$ (hence $|a'|\leq 2$) or 
$|a''|=1$ (hence $|a'|\leq 1$)
\begin{align}
\|&
\chi_1
(t+|x|)^{-1}P_2
\|_{L^{6/5}({\mathbb R}^3)}\\
&
\leq
C
\langle t\rangle^{-2}
\|
|t-r|\partial {\bar Z}^{a'} u(t)
\|_{L^2({\mathbb R}^3)}
\biggl(
\sum_{{|b|+|c|}\atop{+d=1}}
\|
\Omega^b L^c S^d \partial{\bar Z}^{a''} u(t)
\|_{L^3({\mathbb R}^3)}
\biggr)
\nonumber\\
&
\leq
C
\langle t\rangle^{-2}
M_3(u(t))
N_4(u(t)),
\,\,3<t<T.\nonumber
\end{align}
When $|a''|=2$ (hence $|a'|=0$), 
we obtain
\begin{align}
\|&
\chi_1
(t+|x|)^{-1}P_2
\|_{L^{6/5}({\mathbb R}^3)}\\
&
\leq
C
\langle t\rangle^{-2}
\|
|t-r|\partial u(t)
\|_{L^3({\mathbb R}^3)}
\biggl(
\sum_{{|b|+|c|}\atop{+d=1}}
\|
\Omega^b L^c S^d \partial{\bar Z}^{a''} u(t)
\|_{L^2({\mathbb R}^3)}
\biggr)
\nonumber\\
&
\leq
C
\langle t\rangle^{-2}
M_3(u(t))
N_4(u(t)),
\,\,3<t<T.\nonumber
\end{align}
Let us turn our attention to the estimate of 
$\chi_2 (t+|x|)^{-1} P_i$ 
$(i=1,2)$ for $t>0$. 
When $|a''|=0$ (hence $|a'|\leq 2$) or 
$|a''|=1$ (hence $|a'|\leq 1$), 
we get by using 
(\ref{lem311}), (\ref{lem312}), and (\ref{lem313})
\begin{align}
\|&
\chi_2
(t+|x|)^{-1}P_1
\|_{L_r^1 L_\omega^{4/3}({\mathbb R}^3)}\\
&
\leq
C
\langle t\rangle^{-1}
\biggl(
\sum_{{|b|+|c|}\atop{+d=1}}
\|
\Omega^b L^c S^d{\bar Z}^{a'} u(t)
\|_{L^2({\mathbb R}^3)}
\biggr)
\|
\partial^2 {\bar Z}^{a''}u(t)
\|_{L_r^2 L_\omega^4({\mathbb R}^3)}\nonumber\\
&
\leq
C\langle t\rangle^{-1}
M_3(u(t))
\biggl(
\sum_{|a|\leq 2}
\|\partial{\bar Z}^a u(t)\|_{L_r^2 L_\omega^4({\mathbb R}^3)}
+
\|u(t)\|_{L_r^2 L_\omega^4({\mathbb R}^3)}
\biggr)\nonumber\\
&
\leq
C
\langle t\rangle^{-1}
M_3(u(t))
\bigl(
N_4(u(t))
+
X_1(u(t))
\bigr).\nonumber
\end{align}
When $|a''|=2$ (and hence $|a'|=0$), 
we get by (\ref{lem314})
\begin{align}
\|&
\chi_2
(t+|x|)^{-1}P_1
\|_{L_r^1 L_\omega^{4/3}({\mathbb R}^3)}\\
&
\leq
C
\langle t\rangle^{-1}
\biggl(
\sum_{{|b|+|c|}\atop{+d=1}}
\|
\Omega^b L^c S^d u(t)
\|_{L_r^2 L_\omega^4({\mathbb R}^3)}
\biggr)
\|
\partial^2 {\bar Z}^{a''}u(t)
\|_{L^2({\mathbb R}^3)}\nonumber\\
&
\leq
C\langle t\rangle^{-1}
M_2(u(t))
\bigl(
N_4(u(t))
+
\|u(t)\|_{L^2({\mathbb R}^3)}
\bigr)\nonumber\\
&
\leq
C
\langle t\rangle^{-1}
M_2(u(t))
\bigl(
N_4(u(t))
+
X_0(u(t))
\bigr).\nonumber
\end{align}
As for $P_2$, we get
\begin{align}
\|&
\chi_2
(t+|x|)^{-1}
P_2
\|_{L_r^1 L_\omega^{4/3}({\mathbb R}^3)}\\
&
\leq
C
\langle t\rangle^{-1}
\|
\partial{\bar Z}^{a'} u(t)
\|_{L_r^2 L_\omega^4 ({\mathbb R}^3)}
\sum_{{|b|+|c|}\atop{+d=1}}
\|
\Omega^b L^c S^d \partial{\bar Z}^{a''} u(t)
\|_{L^2 ({\mathbb R}^3)}\nonumber\\
&
\leq
C
\langle t\rangle^{-1}
N_3(u(t))N_4(u(t)).\nonumber
\end{align}

\noindent{\bf$\cdot$Estimate of (\ref{ly236}).} 
As in the estimate of (\ref{ly136}), 
we need to handle the $L^{6/5}$ norm for 
$t<3$ and $t>3$, separately. 
It is easy to get for $|a|\leq 2$
\begin{equation}
\|
\chi_1
{\bar Z}^a F_i^{jk,\alpha\beta}
(\partial_\alpha u_j)
(\partial_\beta u_k)
\|_{L^{6/5}({\mathbb R}^3)}
\leq
C
N_2(u(t))
N_3(u(t)),\,\,0<t<3
\end{equation}
On the other hand, for $t>3$ 
we need to bound 
$(t+|x|)^{-1}P_i$ 
$(i=3,4)$. 
By symmetry, it suffices to treat only 
$(t+|x|)^{-1}P_3$. See (\ref{march7413}). 

When $|a''|=0$ 
(and hence $|a'|\leq 2$) 
or 
$|a''|=1$ 
(and hence $|a'|\leq 1$), 
we get by (\ref{ddd})
\begin{align}
\|&
\chi_1(t+|x|)^{-1}P_3
\|_{L^{6/5}({\mathbb R}^3)}\\
&
\leq
C
\langle t\rangle^{-2}
\biggl(
\sum_{{|b|+|c|}\atop{+d=1}}
\|
\Omega^b L^c S^d 
{\bar Z}^{a'}
u(t)
\|_{L^2 ({\mathbb R}^3)}
\biggr)
\|
|t-r|\partial
{\bar Z}^{a''} u(t)
\|_{L^3({\mathbb R}^3)}\nonumber\\
&
\leq
C
\langle t\rangle^{-2}
M_2(u(t))M_3(u(t)),\,\,3<t<T.\nonumber
\end{align}
When $|a''|=2$, we get
\begin{align}
\|&
\chi_1(t+|x|)^{-1}P_3
\|_{L^{6/5}({\mathbb R}^3)}\\
&
\leq
C
\langle t\rangle^{-2}
\biggl(
\sum_{{|b|+|c|}\atop{+d=1}}
\|
\Omega^b L^c S^d u(t)
\|_{L^3 ({\mathbb R}^3)}
\biggr)
\|
|t-r|\partial
{\bar Z}^{a''} u(t)
\|_{L^2({\mathbb R}^3)}\nonumber\\
&
\leq
C
\langle t\rangle^{-2}
M_2(u(t)M_3(u(t)),\,\,3<t<T.
\nonumber
\end{align}
Moreover, we easily obtain for $t>0$
\begin{align}
\|&
\chi_2(t+|x|)^{-1}P_3
\|_{L_r^1 L_\omega^{4/3}({\mathbb R}^3)}\\
&
\leq
C
\langle t\rangle^{-1}
\biggl(
\sum_{{|b|+|c|}\atop{+d=1}}
\|
\Omega^b L^c S^d {\bar Z}^{a'}u(t)
\|_{L^2 ({\mathbb R}^3)}
\biggr)
\|
\partial
{\bar Z}^{a''} u(t)
\|_{L_r^2 L_\omega^4 ({\mathbb R}^3)}\nonumber\\
&
\leq
C
\langle t\rangle^{-1}
M_3(u(t))N_4(u(t)).\nonumber
\end{align}

\noindent{\bf$\cdot$Estimate of (\ref{ly336}).} 
We obtain the following for $|a|\leq 2$\,:
\begin{align}
\|&
\chi_1
{\bar Z}^a
\bigl(
G_i^{\alpha\beta}(u,\partial u)
\partial_{\alpha\beta}^2 u_i
\bigr)
\|_{L^{6/5}({\mathbb R}^3)}\label{march8516}\\
&
\leq
C
\langle t\rangle^{-2}
\bigl(
M_3(u(t))^2
+
N_3(u(t))^2
\bigr)
\bigl(
N_4(u(t))
+
X_0(u(t))
\bigr),\nonumber\\
\|&
\chi_2
{\bar Z}^a
\bigl(
G_i^{\alpha\beta}(u,\partial u)
\partial_{\alpha\beta}^2 u_i
\bigr)
\|_{L_r^1 L_\omega^{4/3}({\mathbb R}^3)}\label{march8517}\\
&
\leq
C\langle t\rangle^{-1}
\bigl(
N_3(u(t))^2
+
X_2(u(t))^2
\bigr)
\bigl(
N_4(u(t))
+
X_2(u(t))
\bigr).\nonumber
\end{align}
For the proof of (\ref{march8516})--(\ref{march8517}), 
it suffices to explain how to deal with such a typical term as 
$u_j({\bar Z}^a u_k)\partial^2 u_i$ 
$(|a|=2)$. 
Using (\ref{lem311})--(\ref{lem312}) and (\ref{3ineq1}), 
we get
\begin{align}
\|&
\chi_1 
u_j(t)({\bar Z}^a u_k(t))\partial^2 u_i(t)
\|_{L^{6/5}({\mathbb R}^3)}\\
&
\leq
C\langle t\rangle^{-2}
\|
\langle t-r \rangle u_j(t)
\|_{L^6({\mathbb R}^3)}
\|
\langle t-r \rangle{\bar Z}^a u_k(t)
\|_{L^6({\mathbb R}^3)}
\|
\partial^2 u_i(t)
\|_{L^2({\mathbb R}^3)}\nonumber\\
&
\leq
C\langle t\rangle^{-2}
\bigl(
M_3(u(t))
+
N_3(u(t))
\bigr)^2
\bigl(
N_2(u(t))
+
X_0(u(t))
\bigr).\nonumber
\end{align}
We also obtain by (\ref{sob24}), (\ref{lem311})--(\ref{lem312})
\begin{align}
\|&
\chi_2
u_j(t)({\bar Z}^a u_k(t))\partial^2 u_i(t)
\|_{L_r^1 L_\omega^{4/3} ({\mathbb R}^3)}\\
&
\leq
C\langle t\rangle^{-1}
\|
ru_j(t)
\|_{L_r^\infty L_\omega^4 ({\mathbb R}^3)}
\|
{\bar Z}^a u_k(t)
\|_{L^2 ({\mathbb R}^3)}
\|
\partial^2 u_i(t)
\|_{L_r^2 L_\omega^\infty ({\mathbb R}^3)}\nonumber\\
&
\leq
C\langle t\rangle^{-1}
\bigl(
N_1(u(t))
+
X_1(u(t))
\bigr)
X_2(u(t))
\bigl(
N_4(u(t))
+
X_2(u(t))
\bigr).\nonumber
\end{align}

\noindent{\bf $\cdot$Estimate of (\ref{ly436}).} 
We can prove for $|a|\leq 2$
\begin{align}
\|&
\chi_1
{\bar Z}^a H_i(u,\partial u)
\|_{L^{6/5}({\mathbb R}^3)}\\
&
\leq
C\langle t \rangle^{-2}
\bigl(
M_2(u(t))
+
N_2(u(t))
\bigr)^2
\bigl(
N_3(u(t))
+
X_2(u(t))
\bigr),\nonumber\\
\|&
\chi_2
{\bar Z}^a H_i(u,\partial u)
\|_{L_r^1 L_\omega^{4/3}({\mathbb R}^3)}
\leq
C\langle t\rangle^{-1}
\bigl(
N_3(u(t))
+
X_2(u(t))
\bigr)^3.
\end{align}
The proof is similar to what we have done above. 
We may therefore omit it. 

Obviously, the estimate (\ref{2019estimatex2}) 
follows from what we have just obtained above. 
The proof of Proposition \ref{propositionx2} 
has been finished.
\section{Space-time $L^2$ estimate}\label{sectionspacetimel2norm}
Recall the definition of $L(v(t))$ 
(see (\ref{lv311322})). 
The purpose of this section is to prove the following proposition. 
\begin{proposition}\label{propositionstl2}
Suppose that initial data $(\ref{data1})$ is smooth and compactly supported, 
and suppose that the local solution $u$ satisfies $(\ref{lem311})$ in 
some interval $(0,T)$. 
Then the following inequality holds for all $t\in (0,T):$ 
\begin{align}\label{2019estimatestl2}
(&1+t)^{-1/2}
\int_0^t L(u_i(\tau))^2 d\tau\\
&
\leq
C
\sum_{{a|+|c|+d\leq 3}\atop{|c|+d\leq 1}}
\|
(\partial{\bar Z}^a L^c S^d u_i)(0)
\|_{L^2({\mathbb R}^3)}^2\nonumber\\
&
\hspace{0.2cm}
+
C\int_0^t
\langle\tau\rangle^{-1}
\bigl(
M_3(u(\tau))+N_4(u(\tau))
\bigr)
L(u(\tau))^2 d\tau\nonumber\\
&
\hspace{0.2cm}
+
C\int_0^t
\langle \tau\rangle^{-3/2}
\bigl(
M_3(u(\tau))
+
N_4(u(\tau))
\bigr)
N_4(u(\tau))^2 d\tau\nonumber\\
&
\hspace{0.2cm}
+
C\int_0^t
\langle \tau\rangle^{-1+\eta}
\bigl(
M_3(u(\tau))
+
N_4(u(\tau))
\bigr)
N_4(u(\tau))G(u(\tau)) d\tau\nonumber\\
&
\hspace{0.2cm}
+
C\int_0^t
\langle \tau\rangle^{-1}
\bigl(
M_3(u(\tau))^2
+
N_4(u(\tau))^2
\bigr)
L(u(\tau))^2d\tau\nonumber\\
&
\hspace{0.2cm}
+
C\int_0^t
\langle \tau\rangle^{-2}
\bigl(
M_3(u(\tau))^3
+
N_4(u(\tau))^3
\bigr)
N_4(u(\tau))d\tau\nonumber\\
&
\hspace{0.2cm}
+
C\int_0^t
\langle \tau\rangle^{-2}
\bigl(
M_3(u(\tau))^2
+
N_4(u(\tau))^2
\bigr)
N_4(u(\tau))^2d\tau
\nonumber
\\
&
\hspace{0.2cm}
+
C\int_0^t
\langle\tau\rangle^{-1}
N_4(u(\tau))^2
\bigl(
N_4(u(\tau))
+
M_1(u(\tau))
+
X_2(u(\tau))
\bigr)d\tau\nonumber\\
&
\hspace{0.2cm}
+
C
\int_0^t
\langle\tau\rangle^{-1}
\bigl(
N_3(u(\tau))
+
M_1(u(\tau))
+
X_2(u(\tau))
\bigr)
L(u(\tau))^2
d\tau.\nonumber
\end{align}
\end{proposition}
For the proof of this proposition, 
it suffices to explain how to bound 
\begin{align}
(&1+t)^{-1/2}
\|
|x|^{-5/4}
{\bar Z}^{a}Su_i
\|_{L^2((0,t)\times{\mathbb R}^3)}^2\\
&
+
(1+t)^{-1/2}
\|
|x|^{-1/4}
\partial{\bar Z}^{a}Su_i
\|_{L^2((0,t)\times{\mathbb R}^3)}^2,\quad |a|=2,\nonumber
\end{align}
because the others can be handled similarly. 
Naturally, 
we rely upon Lemma \ref{spacetimeL2}, 
with $\mu=1/4$, 
$u={\bar Z}^a S u_i$, 
$h^{\alpha\beta}
=F_i^{j,\alpha\beta\gamma}\partial_\gamma u_j+G_i^{\alpha\beta}(u,\partial u)$ 
(see (\ref{nlt1}), (\ref{nlt2}), (\ref{pdef})). 
\subsection{Estimate of $\int|\partial u||Pu|dx$}~\\
We start with $\int|\partial u||Pu|dx$, that is, 
$\|
(\partial{\bar Z}^a S u_i)
(P{\bar Z}^a S u_i)
\|_{L^1({\mathbb R}^3)}$. 
Actually, this norm has been already estimated in Section \ref{sect3}. 
See (\ref{dec1}), (\ref{dec5})--(\ref{july6}), 
(\ref{2019aug31456})--(\ref{ji52019july261736}), 
and (\ref{2019aug31459})--(\ref{1818}). 
We thus see that 
it is bounded from above by the sum of 
the third, the fourth, \dots, and the eighth terms 
on the right-hand side of (\ref{genergyest}).  
\subsection{Estimate of $\int |x|^{-1/2}\langle x\rangle^{-1/2}|u||Pu|dx$}~\\
Next, we consider $\int |x|^{-1/2}\langle x\rangle^{-1/2}|u||Pu|dx$, that is, 
$\|
|x|^{-1/2}\langle x\rangle^{-1/2}
({\bar Z}^a S u_i)
(P {\bar Z}^a S u_i)
\|_{L^1({\mathbb R}^3)}$. 
Due to the simple inequality 
$|x|^{-1/2}\langle x\rangle^{-1/2}\leq C|x|^{-1}$, 
we can rely upon the Hardy inequality 
or the norm $\||x|^{-5/4} {\bar Z}^a S u_i(t)\|_{L^2({\mathbb R}^3)}$ 
(see (\ref{lv311322}) for the definition of the norm $L(u(t))$), 
which implies that we have only to repeat the same discussion as in Subsection 6.1. 
\subsection{Estimate of $\int |\partial h||\partial u|^2 dx$}\label{subsect63}~\\
Our next concern is to bound $\int |\partial h||\partial u|^2 dx$, 
that is ,
$
\|
(\partial h)
(\partial{\bar Z}^a S u_i)^2
\|_{L^1({\mathbb R}^3)}, 
$
with 
$h=F_i^{j,\alpha\beta\gamma}\partial_\gamma u_j+G_i^{\alpha\beta}(u,\partial u)$. 
Using (\ref{lem311}) and (\ref{lem312}), 
we easily obtain 
\begin{align}
\|&
(\partial h(t))
(\partial{\bar Z}^a S u_i(t))^2
\|_{L^1({\mathbb R}^3)}\\
&
\leq
C\sum_{|b|\leq 1}
\|
(\partial\partial_x^b u(t))
(\partial{\bar Z}^a S u_i(t))^2
\|_{L^1({\mathbb R}^3)}
+
C\|
u(t) (\partial{\bar Z}^a S u_i(t))^2
\|_{L^1({\mathbb R}^3)}\nonumber\\
&
\leq
C
\langle t\rangle^{-1}
\biggl(
\sum_{|b|\leq 1}
\|
\langle t-r\rangle\partial\partial_x^b u(t)
\|_{L^\infty({\mathbb R}^3)}
+
\|
\langle t-r\rangle u(t)
\|_{L^\infty({\mathbb R}^3)}\nonumber\\
&
\hspace{2cm}
+
\sum_{|b|\leq 1}
\|
r\partial\partial_x^b u(t)
\|_{L^\infty({\mathbb R}^3)}
+
\|
r u(t)
\|_{L^\infty({\mathbb R}^3)}
\biggr)N_4(u(t))^2\nonumber\\
&
\leq
C
\langle t\rangle^{-1}
\bigl(
N_4(u(t))
+
M_1(u(t))
+
N_2(u(t))^{1/2}
X_2(u(t))^{1/2}
\bigr)N_4(u(t))^2\nonumber\\
&
\leq
C
\langle t\rangle^{-1}
\bigl(
N_4(u(t))
+
M_1(u(t))
+
X_2(u(t))
\bigr)N_4(u(t))^2.\nonumber
\end{align}
\subsection{Estimate of $\int |x|^{-1/2}\langle x\rangle^{-1/2}
|\partial h||u\partial u|dx$}~\\
We next consider 
$\int |x|^{-1/2}\langle x\rangle^{-1/2}
|\partial h||u\partial u|dx$, 
which is 
$
\|
|x|^{-1/2}\langle x\rangle^{-1/2}
(\partial h)
({\bar Z}^a S u_i)
\partial{\bar Z}^a S u_i
\|_{L^1({\mathbb R}^3)}$. 
We rely upon the Hardy inequality to get
\begin{align*}
\|&
|x|^{-1/2}\langle x\rangle^{-1/2}
(\partial h(t))
({\bar Z}^a S u_i(t))
\partial{\bar Z}^a S u_i(t)
\|_{L^1({\mathbb R}^3)}\\
&
\leq
C
\|
\partial h(t)
\|_{L^\infty({\mathbb R}^3)}
\|
|x|^{-1}{\bar Z}^a S u_i(t)
\|_{L^2({\mathbb R}^3)}
\|
\partial{\bar Z}^a S u_i(t)
\|_{L^2({\mathbb R}^3)}
\leq
C\|
\partial h(t)
\|_{L^\infty({\mathbb R}^3)}
N_4(u(t))^2,
\end{align*}
which implies that we have only to repeat the same argument as in 
Subsection 6.3. 
\subsection{Estimate of $\int |x|^{-1/2}\langle x\rangle^{-1/2}
|h||\partial u|^2dx$}~\\
Next, let us consider 
$\int |x|^{-1/2}\langle x\rangle^{-1/2}|h||\partial u|^2dx$, 
that is, 
$
\|
|x|^{-1/2}\langle x\rangle^{-1/2}
h
(\partial{\bar Z}^a S u_i)^2
\|_{L^1({\mathbb R}^3)}
$. 
Recalling that 
$h=F_i^{j,\alpha\beta\gamma}\partial_\gamma u_j+G_i^{\alpha\beta}(u,\partial u)$, 
we get 
by (\ref{3ineq3}), (\ref{3ineq4}), and (\ref{sob25})
\begin{align}
\|&
|x|^{-1/2}\langle x\rangle^{-1/2}
|h(t)|
(\partial{\bar Z}^a S u_i(t))^2
\|_{L^1({\mathbb R}^3)}\\
&
\leq
C
\langle t\rangle^{-1}
\bigl(
\|
\langle t-r\rangle\partial u(t)
\|_{L^\infty({\mathbb R}^3)}
+
\|
\langle t-r\rangle u(t)
\|_{L^\infty({\mathbb R}^3)}\nonumber\\
&
\hspace{2cm}
+
\|
r\partial u(t)
\|_{L^\infty({\mathbb R}^3)}
+
\|
r u(t)
\|_{L^\infty({\mathbb R}^3)}
\bigr)
\|
|x|^{-1/4}
\partial{\bar Z}^a S u_i(t)
\|_{L^2 ({\mathbb R}^3)}^2\nonumber\\
&
\leq
C
\langle t\rangle^{-1}
\bigl(
N_3(u(t))
+
M_1(u(t))
+
X_2(u(t))
\bigr)
L(u(t))^2.\nonumber
\end{align}
\subsection{Estimate of $\int |x|^{-3/2}\langle x\rangle^{-1/2}
|h||u\partial u|dx$}\label{subsect66}~\\
Finally, 
we consider 
$\int |x|^{-3/2}\langle x\rangle^{-1/2}
|h||u\partial u|dx$, 
that is, 
$$
\|
|x|^{-3/2}\langle x\rangle^{-1/2}
|h|
({\bar Z}^a S u_i(t))
(\partial{\bar Z}^a S u_i(t))
\|_{L^1({\mathbb R}^3)}.
$$ 
Recalling the definition of $L(u(t))$ 
(see (\ref{lv311322})) and 
using the norm 
$\||x|^{-5/4} {\bar Z}^a S u_i(t)\|_{L^2({\mathbb R}^3)}$, 
we can bound it in the same way as in Subsection 6.5.

Now we are in a position to complete the proof of 
Proposition \ref{propositionstl2}. 
Obviously, the estimate (\ref{2019estimatestl2}) 
is a direct consequence of what we have just obtained above. 
The proof has been finished.
\section{Proof of Theorem \ref{maintheorem}}
We are in a position to prove Theorem \ref{maintheorem}, 
firstly for smooth data with compact support. 
Our proof of global existence uses the method of continuity, 
and the important property (\ref{continuitymethod2019}) is easier to show 
when smooth data have compact support.

We use the notation 
 \begin{align}
 &{\mathcal N}_T(w):=\sup_{0<t<T}N_4(w(t)),\\
 &{\mathcal M}_T(w)
 :=
 \sup_{0<t<T}\langle t\rangle^{-\delta}M_3(w(t)),\\
 &{\mathcal X}_T(w)
 :=
 \sup_{0<t<T}X_2(w(t)),\\
 &{\mathcal G}_T(w):=
 \biggl(\int_0^T G(w(t))^2 dt\biggr)^{1/2},\\
 &{\mathcal L}_T(w)
 :=
 \sup_{0<t<T}
 \langle t\rangle^{-(1/4)-\delta/2}
 \biggl(\int_0^t L(w(\tau))^2 d\tau
 \biggr)^{1/2}.
 \end{align}
Here, $0<\delta<1/6$. 
The proof of global existence basically consists of two steps. 
We firstly show that the estimate 
$\max
\{
{\mathcal N}_{T^*}(u),\,
{\mathcal M}_{T^*}(u),\,
{\mathcal X}_{T^*}(u)
\}
\leq
2C_0D(f,g)
$
implies 
$
\sup
\{\langle\!\langle u(t)\rangle\!\rangle:\,t\in(0,T^*)\}
\leq
\varepsilon_1^*$, 
and we secondly show that the latter implies 
the improved estimate 
$
\max
\{
{\mathcal N}_{T^*}(u),\,
{\mathcal M}_{T^*}(u),\,
{\mathcal X}_{T^*}(u)
\}
\leq
\sqrt{3}C_0D(f,g)$. 
See (\ref{langlerangle}) for 
$\langle\!\langle u(t)\rangle\!\rangle$, 
and see (\ref{small7}) for $D(f,g)$. 
See (\ref{3quantities411}) and (\ref{lem311}) for $T^*$ and $\varepsilon_1^*$. 
We will set the constant $C_0$ below (see (\ref{2019c0aug10})). 
The last estimate, together with the standard 
local existence theorem, implies existence of 
global solutions.

We first note that, 
using the idea of decomposing the time interval $(1,T)$ 
dyadically as in Sogge \cite[p.\,363]{Sogge2003} 
(see also \cite{HZ2019}, (161) and (125)), 
we get
\begin{equation}\label{dyadic11}
\int_0^T
\langle \tau\rangle^{-1+2\delta}
L(w(\tau))^2
d\tau
\leq
C{\mathcal L}_T(w)^2
\end{equation}
owing to $\delta<1/6$, 
\begin{equation}\label{dyadic22}
\int_0^T
\langle \tau\rangle^{-1+\eta+\delta}
G(w(\tau))
d\tau
\leq
C{\mathcal G}_T(w)
\end{equation}
owing to $\eta<1/3$ (and hence $\eta+\delta<1/2$). 
We also note that 
it follows from (\ref{sob22}), (\ref{sob25}) and 
the Sobolev embedding 
$H^2\hookrightarrow L^\infty$ that 
\begin{equation}\label{sizesize408}
\sup_{t\in (0,T)}
\langle\!\langle w(t)\rangle\!\rangle
\leq
C^*
\bigl(
{\mathcal N}_T(w)
+
{\mathcal M}_T(w)
+
{\mathcal X}_T(w)
\bigr)
\end{equation}
for a constant $C^*>0$ independent of $T$. 
Therefore, using (\ref{small32}), 
Propositions \ref{ghostenergyestimate}, 
\ref{propositionm3}, \ref{propositionx2}, and \ref{propositionstl2} 
together with (\ref{dyadic11})--(\ref{sizesize408}) 
and the Young inequality, 
we see that, if the local solution defined for 
$(t,x)\in [0,T)\times{\mathbb R}^3$ satisfies 
\begin{equation}\label{sizecondition411}
C^*
\bigl(
{\mathcal N}_T(u)
+
{\mathcal M}_T(u)
+
{\mathcal X}_T(u)
\bigr)
\leq
\varepsilon_1^*,
\end{equation}
then we have 
\begin{align}\label{estimateng413}
&{\mathcal N}_T(u)^2
+
{\mathcal G}_T(u)^2
+
{\mathcal L}_T(u)^2
\\
&
\leq
C_{11}D(f,g)^2
\bigl(
1+D(f,g)^4
\bigr)\nonumber\\
&
+
C
\bigl(
{\mathcal N}_T(u)^3
+
{\mathcal M}_T(u)^3
+
{\mathcal X}_T(u)^3
\bigr)
+
C
\bigl(
{\mathcal N}_T(u)^4
+
{\mathcal M}_T(u)^4
\bigr)\nonumber\\
&
+
C_{12}
\bigl(
{\mathcal N}_T(u)
+
{\mathcal M}_T(u)
+
{\mathcal X}_T(u)
\bigr)
{\mathcal L}_T(u)^2
+
C_{13}
\bigl(
{\mathcal N}_T(u)^2
+
{\mathcal M}_T(u)^2
\bigr)
{\mathcal G}_T(u)\nonumber\\
&
+
C_{14}
\bigl(
{\mathcal N}_T(u)^2
+
{\mathcal M}_T(u)^2
\bigr)
{\mathcal L}_T(u)^2,\nonumber
\end{align}
\begin{align}\label{estimatemtu414}
{\mathcal M}_T(u)
\leq&
C_{21}D(f,g)
+
C
\bigl(
{\mathcal N}_T(u)^2
+
{\mathcal M}_T(u)^2
\bigr)\\
&
+C
\bigl(
{\mathcal N}_T(u)^3
+
{\mathcal M}_T(u)^3
+
{\mathcal X}_T(u)^3
\bigr)
+
C
{\mathcal M}_T(u)
{\mathcal N}_T(u)^3,\nonumber
\end{align}
\begin{align}\label{estimatextu414}
{\mathcal X}_T(u)
\leq&
C_{31}D(f,g)
+
C
\bigl(
{\mathcal N}_T(u)^2
+
{\mathcal M}_T(u)^2
+
{\mathcal X}_T(u)^2
\bigr)\\
&
+
C
\bigl(
{\mathcal N}_T(u)^3
+
{\mathcal M}_T(u)^3
+
{\mathcal X}_T(u)^3
\bigr).\nonumber
\end{align}
We remark that 
in (\ref{estimatemtu414}), 
the equality 
$r\partial_k=\Lambda+\omega_j\Omega_{jk}$ 
has been used. 
We set 
\begin{equation}\label{2019c0aug10}
C_0:=\max\{C_d,\,\sqrt{C_{11}},\,C_{21},\,C_{31}\}.
\end{equation} 
(See (\ref{small32}) for $C_d$.)

Recall that due to the size condition (\ref{dontforgetsmall}) 
(see also (\ref{small31})), 
we enjoy a unique solution at least for a short time interval, say, $[0,T_*)$. 
By virtue of the finite speed of propagation, 
it is easy to observe the important property that 
this local solution satisfies 
\begin{equation}\label{continuitymethod2019}
N_4(u(t)),\,M_3(u(t)),\,X_2(u(t))
\in
C([0,T_*)),
\end{equation}
which implies that 
\begin{equation}
\max\{N_4(u(t)),\,(1+t)^{-\delta}M_3(u(t)),\,X_2(u(t))\}
\leq
2C_dD(f,g)
\end{equation}
at least for a short time interval, say, $[0,{\hat T})$ 
with ${\hat T}\leq T_*$. See (\ref{small32}). 
For the given data $(f_i,g_i)\in 
C_0^\infty({\mathbb R}^3)\times C_0^\infty({\mathbb R}^3)$ 
$(i=1,\dots,N)$ 
satisfying (\ref{small31}), 
we therefore have the non-empty set 
$\{T>0\,:$ There exists a unique smooth solution $u(t,x)$ to (\ref{eq1}) 
defined for all $(t,x)\in [0,T)\times{\mathbb R}^3$ 
satisfying 
\begin{equation}\label{3quantities411}
\max\{N_4(u(t)),\,(1+t)^{-\delta}M_3(u(t)),\,X_2(u(t))\}
\leq
2C_0D(f,g)
\end{equation}
for all $t\in [0,T)$\}. 
Define $T^*\in (0,\infty]$ as the supremum of this non-empty set. 
(Readers are advised not to confuse it with $T_*$.)

Let $T'$ be an arbitrary number such that $T'<T^*$. 
Since $T'$ is finite and 
$u$ is a smooth solution defined 
for all $(t,x)\in [0,T^*)\times{\mathbb R}^3$ 
with 
${\rm supp}\,u(t,\cdot)\subset
\{x\in{\mathbb R}^3\,:\,|x|<t+R\}$ 
for some $R>0$, 
we can easily verify 
${\mathcal L}_{T'}(u)<\infty$ 
by using the Hardy-type inequality. 
It is also easy to verify 
${\mathcal G}_{T'}(u)<\infty$. 
We are in a position to show that 
the inequality (\ref{estimateng413}) holds for $T=T'$ 
and the last three terms on its right-hand side 
can be absorbed into its left-hand side. 
(These terms are allowed to move to the left-hand side, 
thanks to ${\mathcal L}_{T'}(u),\,{\mathcal G}_{T'}(u)<\infty$.)
We start with the inequality 
\begin{equation}\label{estimatemaxmax411}
\max\{{\mathcal N}_{T'}(u),\,
{\mathcal M}_{T'}(u),\,
{\mathcal X}_{T'}(u)\}
\leq
2C_0D(f,g),
\end{equation}
which holds owing to (\ref{3quantities411}) 
and the definition of $T^*$. 
This inequality (\ref{estimatemaxmax411}), 
combined with the size condition (\ref{small31}), 
implies (\ref{sizecondition411}) with $T=T'$. 
Therefore, we see by (\ref{sizesize408}) that 
the inequality (\ref{lem311}) is true for all 
$t\in [0,T']$, 
which among others implies that 
the inequality (\ref{estimateng413}) holds for $T=T'$. 
Due to the size condition (\ref{small31}), 
it is possible to carry out the absorption step indicated above, 
which yields 
${\mathcal N}_{T'}(u)^2
\leq
2C_{11}D(f,g)^2
+
C_{15}D(f,g)^3
$ 
for a suitable constant $C_{15}>0$. 
Here, we have used the assumption 
$D(f,g)\leq 1$ and the inequality (\ref{estimatemaxmax411}). 
Just to be sure, we note that 
the second last term on the right-hand side of (\ref{estimateng413}) 
with $T=T'$ has been handled as 
\begin{align*}
C&_{13}
\bigl(
{\mathcal N}_{T'}(u)^2
+
{\mathcal M}_{T'}(u)^2
\bigr)
{\mathcal G}_{T'}(u)\\
&
\leq
2C_{13}(2C_0D(f,g))^2{\mathcal G}_{T'}(u)
\leq
C_{13}^2\bigl(4C_0^2D(f,g)^{3/2}\bigr)^2
+
D(f,g){\mathcal G}_{T'}(u)^2
\end{align*}
and the last term above has moved to the left-hand side 
of (\ref{estimateng413}) with $T=T'$. 
Since $T'(<T^*)$ is arbitrary, 
we finally get
\begin{equation}
{\mathcal N}_{T^*}(u)^2
\leq
3C_{11}D(f,g)^2
\end{equation}
due to the condition (\ref{small31}). 

As for the estimate of ${\mathcal M}_{T^*}(u)$ and ${\mathcal X}_{T^*}(u)$, 
we may start with the inequality 
\begin{equation}\label{maxnmxtu414}
\max\{{\mathcal N}_{T^*}(u),\,{\mathcal M}_{T^*}(u),\,
{\mathcal X}_{T^*}(u)\}
\leq
2C_0D(f,g),
\end{equation}
which is a direct consequence of (\ref{3quantities411}) 
and the definition of $T^*$, 
and combine it with the size condition (\ref{small31}) 
to see that 
the condition (\ref{sizecondition411}) holds 
for $T=T^*$. 
It therefore follows 
directly from (\ref{estimatemtu414}), (\ref{estimatextu414}) 
with $T=T^*$ and (\ref{maxnmxtu414}) that 
\begin{eqnarray}
&
{\mathcal M}_{T^*}(u)
\leq
C_{21}D(f,g)
+
C_{22}D(f,g)^2,\\
&
{\mathcal X}_{T^*}(u)
\leq
C_{31}D(f,g)
+
C_{32}D(f,g)^2
\end{eqnarray}
for suitable constants $C_{22},\,C_{32}>0$. 
Here, we have naturally used the assumption 
$D(f,g)\leq 1$. 
Using the size condition (\ref{small31}), 
we finally obtain 
\begin{equation}\label{202103261506}
{\mathcal M}_{T^*}
\leq
\frac32
C_{21}D(f,g),
\,\,
{\mathcal X}_{T^*}
\leq
\frac32
C_{31}D(f,g).
\end{equation} 
In sum, we have obtained 
\begin{align}\label{keyestimate417}
\max\{{\mathcal N}_{T^*}(u),\,{\mathcal M}_{T^*}(u),\,
{\mathcal X}_{T^*}(u)\}
&\leq
\max
\biggl\{
\sqrt{3C_{11}},\,
\frac32
C_{21},\,
\frac32
C_{31}
\biggr\}
D(f,g)\\
&
\leq
\sqrt{3}
C_0D(f,g).\nonumber
\end{align}
If we suppose $T^*<\infty$, 
then it is possible to show that 
there exists ${\bar T}>T^*$ 
such that 
the unique local solution $u(t,x)$ 
to (\ref{eq1})--(\ref{data1}) 
exists for all 
$(t,x)\in [0,{\bar T})\times{\mathbb R}^3$ 
satisfying (\ref{3quantities411}) 
for all $t\in [0,{\bar T})$. 
This contradicts the definition of $T^*$. 
We therefore have $T^*=\infty$, 
which proves Theorem \ref{maintheorem} 
at least for smooth, compactly supported data. 

It remains to relax the regularity of data 
and eliminate compactness of the support of data. 
We naturally rely upon the standard idea of 
using the mollifier and the cut-off technique. 
It is then possible to show that, 
for any data $(f,g)$ satisfying 
$D(f,g)\leq \varepsilon_0/2$ 
(see (\ref{small31}) for $\varepsilon_0$), 
there exists a sequence of pairs of compactly supported 
smooth ${\mathbb R}^N$-valued functions $(f^{(n)},g^{(n)})$ 
$(n=1,2,\dots)$ 
such that 
$D(f^{(n)},g^{(n)})\leq \varepsilon_0$ 
for sufficiently large $n$, 
and $D(f-f^{(n)},g-g^{(n)})\to 0$ 
as $n\to\infty$. 
It then follows from the above argument that, 
for sufficiently large $n$, 
we enjoy the unique solution $u^{(n)}$ to (\ref{eq1}) with data 
$(f^{(n)},g^{(n)})$ given at $t=0$ 
which satisfies the estimate (\ref{keyestimate417}) 
with $T^*=\infty$. 
Moreover, repeating the arguments that we have done in the previous sections, 
we obtain for $u_{m,n}:=u^{(m)}-u^{(n)}$
\begin{align}
\max\biggl\{&
\sup_{t>0}
N_2(u_{m,n}(t)),\,\,
\sup_{t>0}
\langle t\rangle^{-\delta}
M_1(u_{m,n}(t)),\,\,
\sup_{t>0}
X_0(u_{m,n}(t)),\\
&
\sum_{j=1}^3
\sum_{{|a|+|c|}\atop{+d\leq 1}}
\|
\langle t-r\rangle^{-(1/2)-\eta}
T_j{\bar Z}^a L^c S^d u_{m,n}
\|_{L^2((0,\infty)\times{\mathbb R}^3)},\nonumber\\
&
\sum_{{|a|+|c|}\atop{+d\leq 1}}
\sup_{T>0}
(1+T)^{-(1/4)-\delta/2}
\|
|x|^{-5/4}
{\bar Z}^a L^c S^d u_{m,n}
\|_{L^2((0,T)\times{\mathbb R}^3)},\nonumber\\
&
\sum_{{|a|+|c|}\atop{+d\leq 1}}
\sup_{T>0}
(1+T)^{-(1/4)-\delta/2}
\|
|x|^{-1/4}
\partial{\bar Z}^a L^c S^d u_{m,n}
\|_{L^2((0,T)\times{\mathbb R}^3)}
\biggr\}
\nonumber\\
\leq&
CD(f^{(m)}-f^{(n)},g^{(m)}-g^{(n)}).\nonumber
\end{align}
(Here we are supposed to choose $\varepsilon_0$ smaller than before, 
if necessary.) 
By the standard argument, we see that 
the sequence $\{u^{(n)}\}$ has the limit and it is the solution to 
$(\ref{eq1})$ that we have sought for. 
The proof of Theorem \ref{maintheorem} has been finished.
\section*{Acknowledgments} 
The authors are very grateful to the referee for 
careful reading of the manuscript and helpful comments. 
The first author was supported by 
JSPS KAKENHI Grant Number JP18K03365. 
\bibliographystyle{amsplain}

%
\end{document}